\documentclass[onefignum,onetabnum]{siamonline220329}

\usepackage{cleveref}

\makeatletter
\AtBeginDocument{%
    \def\refstepcounter@optarg[#1]#2{%
        \cref@old@refstepcounter{#2}%
        \cref@constructprefix{#2}{\cref@result}%
        \@ifundefined{cref@#1@alias}%
            {\def\@tempa{#1}}%
            {\def\@tempa{\csname cref@#1@alias\endcsname}}%
        \protected@edef\cref@currentlabel{%
            [\@tempa][\arabic{#2}][\cref@result]%
            \csname p@#2\endcsname\csname the#2\endcsname}}%
}
\makeatother

\usepackage{mathtools}
\usepackage{amsfonts}
\usepackage{array}
\usepackage{placeins}
\usepackage{needspace}

\usepackage{enumitem}
\setlist[enumerate]{leftmargin=.5in}
\setlist[itemize]{leftmargin=.5in}

\newsiamremark{remark}{Remark}
\newsiamremark{assumption}{Assumption}
\crefname{assumption}{Assumption}{Assumptions}
\newsiamremark{example}{Example}
\crefname{example}{Example}{Examples}
\newcommand{\dummy}{{\color{black!50}\bullet}}
\headers{Prior-to-Posterior Stability in the Wasserstein Metric}{L. Cao}

\title{On Prior-to-Posterior Stability in the Wasserstein Metric\\
for Bayesian Inverse Problems\thanks{This work is funded by the U.S. Department of Defense Vannevar Bush Faculty Fellowship held by Andrew M. Stuart under the Office of Naval Research award number N00014-22-1-2790.}}

\author{Lianghao Cao\thanks{Department of Computing and Mathematical Sciences, California Institute of Technology, CA 91125, USA.
  (\email{lianghao@caltech.edu})
  }
}

\usepackage{amsopn}
\usepackage{amssymb}

\begin{document}

\maketitle

\begin{abstract}
Priors in Bayesian inverse problems are often approximated through discretization, hyperparameter estimation, or generative modeling. Understanding how prior approximation errors propagate to the posterior and subsequent predictions is therefore important. In this work, we study the stability of the prior-to-posterior map where both prior and posterior perturbations are measured in the same Wasserstein metric $W_p$, $p\geq1$. We identify verifiable conditions on likelihood regularity and admissible prior classes that ensure uniform, H\"older, and Lipschitz stability. For bounded likelihoods that are uniformly continuous on bounded sets, uniform stability holds over prior classes with uniformly integrable $p$-th moments and a common positive evidence lower bound. With global H\"older regularity of the likelihood and uniform bounds on higher prior moments, a coupling argument leads to a H\"older estimate with a sharp exponent. For $p>1$, a Lipschitz likelihood need not give Lipschitz stability, even for priors with bounded support. We establish Lipschitz stability through an interpolation argument under uniform Poincar\'e bounds and a globally Lipschitz potential with uniformly bounded essential oscillation under the priors. For Gaussian priors with additive Gaussian noise and bounded Lipschitz forward models, these estimates give posterior $W_2$ bounds even for mutually singular prior perturbations. Numerical experiments for a Darcy inverse problem illustrate the predicted H\"older and Lipschitz rates and the resulting control of errors in the posterior mean and standard deviation of a Lipschitz quantity of interest.
\end{abstract}



\section{Introduction}

Bayesian inverse problems use noisy, indirect observations to update uncertainty about an unknown parameter. Their solution is a posterior probability distribution, from which quantities of interest (QoIs) are estimated with quantified uncertainties. In practice, the prior distribution in a Bayesian inverse problem often arises from approximations: it may be numerically discretized, its hyperparameters may be estimated, or it may be constructed or learned from samples. It is therefore important to understand how such approximations affect the posterior and downstream QoIs predictions.

In this work, we study the stability of the prior-to-posterior map in Bayesian inverse problems, which is formulated as follows. Let the parameter space $(X,d_X)$ be a complete separable metric space. Let $\mu$ be the prior distribution, i.e., a Borel probability measure on $(X,d_X)$, and let a $\mu$-integrable function $L:X\to[0,\infty)$ be the likelihood for a fixed observation. If the normalizing constant, or evidence, $Z_\mu$ is positive, define
\begin{equation}\label{eq:intro_bayes_map}
  Z_\mu\coloneqq\mu(L),
  \qquad
  \mathcal B_L(\mu)(\mathrm{d}x)\coloneqq Z_\mu^{-1}L(x)\,\mu(\mathrm{d}x).
\end{equation}
Here $\mu(L)\coloneqq\int_X L(x)\,\mu(\mathrm{d}x)$, and $\mathcal B_L$ is the prior-to-posterior map associated with $L$. Given assumptions on the likelihood, a class of admissible priors, and discrepancy measures $d_{\rm in}$ and $d_{\rm out}$, we seek a modulus $g:[0,\infty)\to[0,\infty)$, with $g(r)\to0$ as $r\downarrow0$, such that
\begin{equation}\label{eq:intro_stability_question}
  d_{\rm out}\bigl(\mathcal B_L(\mu),\mathcal B_L(\nu)\bigr) \leq g\bigl(d_{\rm in}(\mu,\nu)\bigr),
\end{equation}
for all $\mu$ and $\nu$ in the admissible prior class. For example, $g(r)=Cr$ gives Lipschitz stability. We seek verifiable conditions on the likelihood and the prior class that gives uniform stability, H\"older stability, or Lipschitz stability of the prior-to-posterior map.

We consider the setting in which the Wasserstein metric $W_p$ with $p\geq 1$ quantifies both prior and posterior discrepancies, that is, $d_{\rm in}=d_{\rm out}=W_p$ in \eqref{eq:intro_stability_question}. Let $\mathcal{P}_p(X)$ denote the set of Borel probability measures on $X$ with finite $p$-th moments. For $\mu,\nu\in\mathcal{P}_p(X)$, the $p$-Wasserstein metric is defined by
\begin{equation*}
W_p(\mu,\nu)\coloneqq\left(\inf_{\pi\in\Pi(\mu,\nu)}\int_{X\times X}d_X(x,x')^p\,\pi(\mathrm{d}x,\mathrm{d}x')\right)^{1/p},
\end{equation*}
where $\Pi(\mu,\nu)$ is the set of couplings of $\mu$ and $\nu$. Unlike the Hellinger and total variation metrics, which are maximal for mutually singular measures, and the Kullback--Leibler divergence, which is infinite for mutually singular measures, $W_p$ retains information about the geometric separation of such measures. Much of the existing literature on quantitative Wasserstein stability takes $d_{\rm in}=W_2$ or $W_1$ and $d_{\rm out}=W_1$; see, for example, \cite{sprungk2020local}. Here, we instead seek posterior stability in $d_{\rm in}=d_{\rm out}=W_p$, in particular for $p\geq 2$, for the following reasons.

First, this posterior control in $W_p$ for $p\geq 2$ is useful for estimating the uncertainty of QoIs. As an example, for a Lipschitz function $Q\colon X\to\mathbb{R}$ with constant $\operatorname{Lip}(Q)$ and $\mu,\nu\in \mathcal P_p(X)$ for $p\geq 1$, we have
\begin{equation}\label{eq:mean_bound_wasserstein}
\left|\mu(Q)-\nu(Q)\right|
\leq
\operatorname{Lip}(Q)W_p(\mu,\nu).
\end{equation}
Moreover, we have the following bound for the standard deviation for $p\geq 2$:
\begin{equation}\label{eq:sd_bound_wasserstein}
\left|\operatorname{sd}_{\mu}(Q)-\operatorname{sd}_{\nu}(Q)\right|
\leq
\operatorname{Lip}(Q)W_p(\mu,\nu),
\qquad
\operatorname{sd}_{\mu}(Q)
\coloneqq
\left\lVert Q-\mu(Q)\right\rVert_{L^2_{\mu}(X)}.
\end{equation}
Here $L^2_{\mu}(X)$ denotes the space of square-integrable functions with respect to $\mu$.  Thus, $W_p$-stability with $p\geq 2$ controls perturbations in both the means and standard deviations of Lipschitz QoIs, making it particularly relevant in the context of uncertainty quantification.

Second, an estimate with $d_{\rm in}=d_{\rm out}=W_p$ can be applied recursively in sequential data assimilation. At each step, the filtering distribution is propagated through the model to become the next predictive prior. If the model is Lipschitz and the predictive priors remain in the admissible class, the same stability estimate controls each update. An estimate from $W_2$ to $W_1$ does not directly provide the $W_2$ input required at the next step.

\subsection{Contributions and Paper Organization}
\label{sec:intro_contributions}

Our first contribution is a uniform continuity result for the prior-to-posterior map $\mathcal B_L$ in $W_p$ for $p\geq1$. Proposition~\ref{prop:uniform_stability} assumes a bounded likelihood that is uniformly continuous on bounded sets, uniform integrability of the prior $p$-th moments, and a uniform evidence lower bound. Examples show that none of these assumptions can be omitted from this general statement.

Our second contribution is a $W_p$-H\"older stability result for $p\geq 1$ over an admissible prior class with bounded moments. This result is achieved assuming a bounded likelihood with global $\beta$-H\"older continuity with $\beta\in(0,1]$, over an admissible prior class that satisfies a uniform $q$-moment bound on the prior with $q>p$, and a uniform evidence lower bound. Specifically, Theorem~\ref{thm:holder_stability} gives
\begin{equation}\label{eq:intro_headline}
  W_p\bigl(\mathcal B_L(\mu),\mathcal B_L(\nu)\bigr)
  \leq C_{\rm stab}\,
  W_p(\mu,\nu)^{\alpha},
  \qquad
  \alpha\coloneqq\min\left\{\frac{\beta}{p},1-\frac{p}{q}\right\},
\end{equation}
for a constant $C_{\rm stab}>0$. Examples show that local Lipschitz continuity of the likelihood does not suffice for a uniform power modulus and that the exponent is sharp under the stated assumptions.

Our third contribution is a $W_p$-Lipschitz stability result for $p>1$. Theorem~\ref{thm:lipschitz_stability} assumes a globally Lipschitz potential $\Phi\coloneqq-\log L$, uniformly bounded $L^{p/(p-1)}$-Poincar\'e constants of the priors, and uniformly bounded essential oscillations of the potential under the priors. The proof reweights a $W_p$-optimal coupling through a geometric interpolation of the likelihood and controls $W_p$-distances along the resulting curve of marginals. As a special case, Corollary~\ref{cor:gaussian_prior_stability} gives a $W_2$-Lipschitz estimate for Gaussian priors and bounded Lipschitz forward maps with additive Gaussian noise. We illustrate these estimates in a two-layer Darcy flow inverse problem: linear prior-to-posterior sensitivity is observed for Gaussian priors,  while two-point priors lead to the square-root sensitivity rate predicted by the H\"older estimate.

The rest of the paper is organized as follows. This introduction section ends with a discussion of related work. Section~\ref{sec:preliminary} introduces the mathematical setting and reviews existing results. Section~\ref{sec:uniform_stability} establishes uniform $W_p$-stability and gives examples digesting the assumptions. Section~\ref{sec:holder_stability} proves the $W_p$-H\"older stability estimate and provide examples on the assumptions and the sharpness of the exponent. Section~\ref{sec:lipschitz_stability} proves $W_p$-Lipschitz stability and applies the result to problems with Gaussian priors and additive Gaussian noise. Section~\ref{sec:numerical_example} illustrates these estimates in a two-layer Darcy inverse problem. Section~\ref{sec:conclusion} summarizes the results and discusses future directions. Appendices~\ref{app:proof_of_posterior_coupling_estimate} and~\ref{app:supplementary_lemmas} contain the auxiliary lemmas used in the stability proofs. Appendix~\ref{app:darcy_details} provides details of the numerical example.

\subsection{Related Work}

Sensitivity to the prior and likelihood has a long history in robust Bayesian
analysis \cite{berger1994overview}. Results on Bayesian brittleness also show that weakly specified model classes can lead to instability \cite{owhadi2015brittleness}. The foundational well-posedness and approximation results often fix the prior, perturb the data or forward model, and control the posterior in the Hellinger metric \cite{stuart2010inverse,cotter2010approximation}. This framework also supports quantitative approximation of forward maps and surrogate likelihoods \cite{stuart2018posterior}, and was subsequently extended to broad non-Gaussian and heavy-tailed prior classes \cite{hosseini2017heavy,hosseini2017exponential}. Latz established qualitative data-to-posterior stability in several metrics, including $W_p$, the $p$-Wasserstein metric with $p\geq1$~\cite{latz2020wellposedness}.

Prior perturbations were treated by Sprungk, who proved qualitative $W_p$ stability of $\mathcal{B}_L$ and $W_1$-Lipschitz stability when the ground metric is bounded, and the likelihood is globally Lipschitz \cite{sprungk2020local}; these results are provided in Section~\ref{ssec:existing_results} of this work. Earlier work by Basu et al.\ studied local sensitivity to prior perturbations in total variation and weak-convergence metrics \cite{basu1998stability}, while Basu later considered stability uniformly over observations \cite{basu2000uniform}. In parametric Bayesian statistics, Ghaderinezhad and Ley used Stein kernels to bound the $W_1$ metric between posteriors generated by different scalar priors \cite{ghaderinezhad2019quantification}. More recently, Garbuno-I\~nigo et al.\ introduced likelihood-adapted integral probability metrics that accommodate locally Lipschitz potentials \cite{garbuno2026posterior}. Cvetkovi\'c and Lie derived upper and lower bounds in total variation, Hellinger metric, Kullback--Leibler divergence, and $W_1$ \cite{cvetkovic2025upper}. For bounded ground metrics, Wang and Gorodetsky derive global $W_1$ bounds relative to a fixed prior and use them to propagate approximation errors in Bayesian sequential learning \cite{wang2025global}. Helin et al.\ obtained a data-averaged $W_2$-to-$W_1$ prior-to-posterior bound in the context of optimal experimental design \cite{helin2025bayesian}.

Geometric assumptions and functional inequalities can yield Lipschitz estimates in $W_2$. Peyre proves a $W_2$-stability estimate for measures reweighted by a common compactly supported function and then normalized, under conditions on the original densities and the weight~\cite{peyre2018comparison}. Dolera and Mainini establish $W_2$-Lipschitz continuity of dominated probability kernels under Poincar\'e and Fisher-information bounds~\cite{dolera2023lipschitz}. Their perturbation is in the conditioning variable rather than the prior.

Our work is relevant to data-driven and generative priors \cite{patel2021gan,bohra2022bayesian,dasgupta2024dimension}. Hosseini and Huang propagate prior error in $W_2$ to posterior error in $W_1$ under weighted regularity assumptions and combine the estimate with finite-sample analysis for learned generators \cite{hosseini2026generative}. Trevisan propagated $W_2$ errors between finite-width neural-network output laws and their Gaussian limits into $W_1$ posterior errors \cite{trevisan2023wide}. Nelsen and Yang identify prior-to-posterior stability in $W_2$ as an important research direction for the theory of data-driven priors \cite{nelsen2026operator}.

\section{Preliminary}\label{sec:preliminary}

In this section, we first introduce the setting and notation relevant to this work. Then, we present existing results on $W_p$ stability, in particular qualitative $W_p$ stability and $W_1$-Lipschitz stability with a bounded ground metric, by Sprungk~\cite{sprungk2020local}.

\subsection{Setting and Notation}\label{ssec:setting_notation}

For every bounded and nonzero likelihood considered in this work, we normalize it and assume
\begin{equation}
\label{eq:likelihood_normalization}
 L:X\to[0,1],
 \qquad
 \sup_{x\in X}L(x)=1.
\end{equation}
This entails no loss of generality because $\mathcal B_{cL}=\mathcal B_L$ for every $c>0$. For a strictly positive likelihood, its potential is $\Phi\coloneqq-\log L$, so $L=e^{-\Phi}$.

Let $\mathcal P(X)$ denote the set of Borel probability measures on $X$. Fix a reference point $x_*\in X$ for moments and balls. For normed spaces, we use the norm-induced metric and take $x_*=0$. For $p\geq1$, set
\begin{equation*}
    m_p(\mu)\coloneqq\int_Xd_X(x,x_*)^p\,\mu(\mathrm{d}x),
\qquad
\mathcal P_p(X)\coloneqq\left\{\mu\in\mathcal P(X)\,\middle|\,m_p(\mu)<\infty\right\}.
\end{equation*}

In this work, we consider various notions of continuity and equate stability with continuity. Let $(Y,d_Y)$ be a metric space. For $F:X\to Y$, $A\subseteq X$, and $\beta\in(0,1]$, set
\begin{align*}
 \operatorname{Lip}(F;A)
 &\coloneqq \sup_{\substack{x,x'\in A\\x\neq x'}}
 \frac{d_Y(F(x),F(x'))}{d_X(x,x')}, &
 \operatorname{Hol}_\beta(F;A)
 &\coloneqq \sup_{\substack{x,x'\in A\\x\neq x'}}
 \frac{d_Y(F(x),F(x'))}{d_X(x,x')^\beta}.
\end{align*}
We set both quantities to zero when $A$ has at most one point and omit $A$ when $A=X$. Let $B_X(R)\coloneqq\{x\in X\mid d_X(x,x_*)\leq R\}$. We call $F$ globally Lipschitz (respectively, globally $\beta$-H\"older) when the corresponding constant on $X$ is finite. We call $F$ locally Lipschitz (respectively, locally $\beta$-H\"older) when the constant on $B_X(R)$ is finite for every $R>0$. For normed spaces, this convention agrees with \cite[Assumption~3(iii), p.~359]{dashti2017bayesian}. We call $F$ uniformly continuous on bounded sets if its restriction to $B_X(R)$ is uniformly continuous for every $R>0$. Figure~\ref{fig:continuity_implications} summarizes the implications among these continuity notions.

\begin{figure}[htb]
 \centering
 \includegraphics[width=0.75\textwidth]{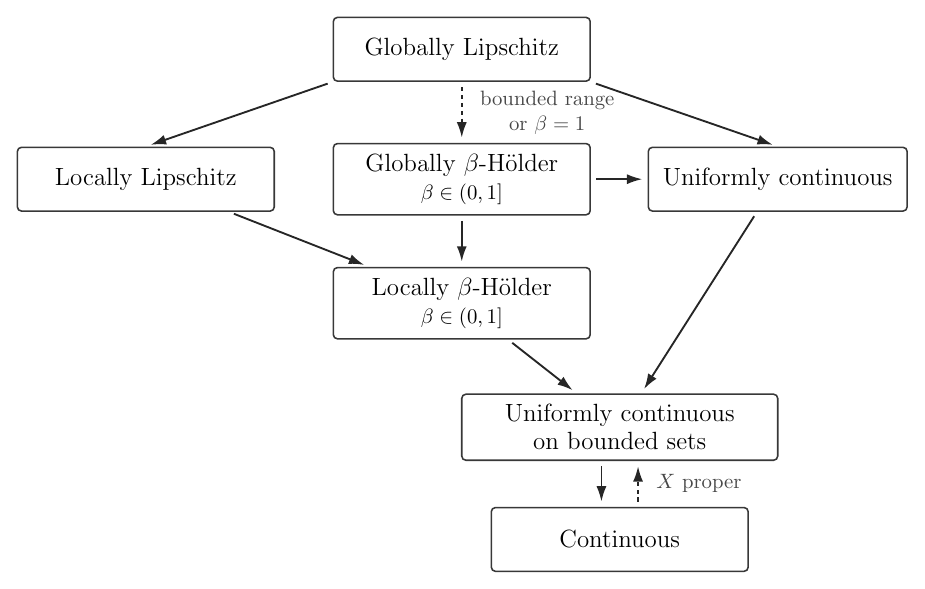}
 \caption{Implications among continuity notions for $F:X\to Y$, with a given $\beta\in(0,1]$. Here, ``local'' means uniform on bounded balls, with constants allowed to depend on the radius. Solid arrows hold unconditionally; labels on dashed arrows give sufficient conditions. Bounded range means $\operatorname{diam}F(X)<\infty$. The space $X$ is proper if its closed bounded balls are compact. For $\beta=1$, the corresponding Lipschitz and H\"older notions coincide.}
 \label{fig:continuity_implications}
\end{figure}

\subsection[Existing Results on Wp Stability]{Existing Results on $W_p$ Stability}\label{ssec:existing_results}
Sprungk \cite[Lemma~16]{sprungk2020local} proves the following result for strictly positive likelihoods. The same proof applies when $L$ is allowed to vanish, provided that the evidence is positive.
\begin{proposition}[Qualitative $W_p$ Stability {\cite[Lemma 16]{sprungk2020local}}]
\label{prop:qualitative_wp}
Let $L:X\to[0,1]$ be continuous and $p\geq 1$. Then $\mathcal B_L$ is
continuous in $W_p$ on $\{\mu\in\mathcal P_p(X)\mid\mu(L)>0\}$.
\end{proposition}

The proof uses the characterization of $W_p$ convergence in terms of weak convergence and convergence of the $p$-th moments. Only positivity of the limiting evidence is required; neither a uniform lower bound on the evidence nor Lipschitz continuity of the likelihood is needed.

The next result by Sprungk uses a \emph{bounded ground metric} $d_D$ on $X$ that induces the same topology as $d_X$ and satisfies
\begin{equation}\label{eq:bounded_ground_metric}
    \sup_{x,x'\in X}d_{D}(x,x')\leq D<\infty,
\end{equation}
for some $D>0$. A standard choice is the truncated metric $d_{D}(x,x')\coloneqq\min\{D,d_X(x,x')\}$. The corresponding $p$-Wasserstein metric is
\begin{equation*}
    W_{p,d_{D}}(\mu,\nu)
 \coloneqq
 \left(\inf_{\pi\in\Pi(\mu,\nu)}
 \int_{X\times X}d_{D}(x,x')^p\,\pi(\mathrm{d}x,\mathrm{d}x')\right)^{1/p}. 
\end{equation*}

\begin{proposition}[$W_1$-Lipschitz Stability with Bounded Ground Metric {\cite[Theorem 15]{sprungk2020local}}]\label{prop:bounded_ground_metric_W1}
Let $d_D$ be as above, and suppose that $L:X\to[0,1]$ is globally Lipschitz with respect to $d_D$, with constant $K$. Let $\mathfrak C\subset\mathcal P(X)$ satisfy $\inf_{\mu\in\mathfrak C}\mu(L)\geq\zeta>0$.
Then, for all $\mu,\nu\in\mathfrak C$,
\begin{equation*}
    W_{1,d_D}\bigl(\mathcal B_L(\mu),\mathcal B_L(\nu)\bigr)
\leq
\zeta^{-2}(1+DK)^2\,W_{1,d_D}(\mu,\nu).
\end{equation*}
\end{proposition}

 With the bounded ground metric, transport beyond distance $D$ has bounded cost, and this estimate does not directly extend to the general case of an unbounded ground metric.

\section{Qualitative Uniform Stability}
\label{sec:uniform_stability}

In this section, we establish uniform $W_p$-stability for $p\geq1$. Lemma~\ref{lem:posterior_coupling_estimate} provides a coupling estimate used here and in the proof of the H\"older stability estimate. Proposition~\ref{prop:uniform_stability} then gives uniform continuity under a bounded likelihood that is uniformly continuous on bounded sets, uniform integrability of the prior $p$-th moments, and a uniform evidence lower bound. The examples in Section~\ref{ssec:uniform_stability_examples} show why these assumptions cannot be omitted in general.

\subsection{Posterior Coupling Estimate}\label{ssec:posterior_coupling_estimate}
We start with a posterior coupling estimate that is critical to both deriving and understanding uniform stability estimates. Given an optimal coupling between the priors, The Bayesian update reweights the two coordinates of this coupling, and the reweighted coupling need not have both posterior measures as its marginals. We provide the following lemma, which controls the $W_p$ metric of the posteriors by retaining the part of the prior optimal coupling compatible with the reweighting and controlling the residual probability mass. The residual-mass construction follows the weighted total-variation coupling argument in \cite[Theorem~6.15]{villani2009optimal}.

\begin{lemma}[Posterior Coupling Estimate]
\label{lem:posterior_coupling_estimate}
Let $L:X\to[0,1]$ be measurable and $p\geq 1$. Let $\mu,\nu\in\mathcal P_p(X)$ satisfy $\mu(L),\nu(L)\geq\zeta>0$ and let $\pi$ be a $W_p$-optimal coupling. Then
\begin{equation}
\label{eq:posterior_coupling_estimate}
\begin{aligned}
 W_p\bigl(\mathcal B_L(\mu),\mathcal B_L(\nu)\bigr)^p
 &\leq\zeta^{-1}W_p(\mu,\nu)^p\\
 &\quad+2^{p-1}\zeta^{-1}\int_{X\times X}
 \bigl(d_X(x,x_*)^p+d_X(x',x_*)^p\bigr)
 |L(x)-L(x')|\,\pi(\mathrm{d}x,\mathrm{d}x')\\
 &\quad+2^{p-1}\zeta^{-2}\left(m_p(\mu)+m_p(\nu)\right)
 \int_{X\times X}|L(x)-L(x')|\,\pi(\mathrm{d}x,\mathrm{d}x').
\end{aligned}
\end{equation}
\end{lemma}

A proof of this lemma is provided in Appendix~\ref{app:proof_of_posterior_coupling_estimate}.

\subsection{Main Result}\label{ssec:uniform_stability}

Based on Lemma~\ref{lem:posterior_coupling_estimate}, the following proposition identifies uniform integrability of the $p$-th moments as a sufficient condition for uniform $W_p$-stability. For a class of priors $\mathfrak C\subset\mathcal P_p(X)$, define its moment-tail function by
\begin{equation*}
\tau_{\mathfrak C,p}(R)\coloneqq\sup_{\mu\in\mathfrak C}\int_{\{d_X(x,x_*)>R\}}d_X(x,x_*)^p\,\mu(\mathrm{d}x).
\end{equation*}

\begin{proposition}[Uniform $W_p$-Stability]
\label{prop:uniform_stability}
Let $L:X\to[0,1]$ be uniformly continuous on bounded sets and $p\geq 1$. Suppose that $\mathfrak C\subset\mathcal P_p(X)$ satisfies $\lim_{R\to \infty}\tau_{\mathfrak C,p}(R)=0$ and $\inf_{\mu\in\mathfrak C}\mu(L)\geq\zeta>0$. 
Then $\mathcal B_L:(\mathfrak C,W_p)\to(\mathcal P_p(X),W_p)$ is uniformly continuous.
\end{proposition}

\begin{proof}[Proof]
Uniform integrability gives $\sup_{\mu\in\mathfrak C}m_p(\mu)<\infty$, since $m_p(\mu)\leq R^p+\tau_{\mathfrak C,p}(R)$. By Lemma~\ref{lem:posterior_coupling_estimate}, it therefore suffices to show that
\begin{equation*}
     \int_{X\times X}
 \bigl(1+d_X(x,x_*)^p+d_X(x',x_*)^p\bigr)
 |L(x)-L(x')|\,\pi(\mathrm{d}x,\mathrm{d}x')
 \longrightarrow0
\end{equation*}
uniformly as $W_p(\mu,\nu)\to0$, where $\pi$ is a $W_p$-optimal coupling of $\mu$ and $\nu$. For $R,r>0$, define
\begin{equation*}
 \omega_{L,R}(r)\coloneqq
 \sup_{\substack{x,x'\in B_X(R)\\d_X(x,x')}\leq r}
 |L(x)-L(x')|.
\end{equation*}
By assumption, $\omega_{L,R}(r)\to0$ as $r\downarrow0$ for every $R>0$. Let

\begin{equation*}
 E_R\coloneqq\{(x,x')\in X\times X\mid d_X(x,x_*)^p+d_X(x',x_*)^p\leq2R^p\}.
\end{equation*}
On
$E_R\cap\{d_X(x,x')\leq r\}$, both points lie in $B_X(2^{1/p}R)$.
Markov's inequality and $|L(x)-L(x')|\leq1$ therefore give
\begin{align*}
&\int_{X\times X}
 \bigl(1+d_X(x,x_*)^p+d_X(x',x_*)^p\bigr)
 |L(x)-L(x')|\,\pi(\mathrm{d}x,\mathrm{d}x')\\
&\quad\leq(1+2R^p)
 \left(\omega_{L,2^{1/p}R}(r)+r^{-p}W_p(\mu,\nu)^p\right)
 +\left(4+2R^{-p}\right)\tau_{\mathfrak C,p}(R).
\end{align*}
The last term bounds the integral over $E_R^c$. Indeed, $1\leq(d_X(x,x_*)^p+d_X(x',x_*)^p)/(2R^p)$ on this set, and
\begin{equation*}
     (d_X(x,x_*)^p+d_X(x',x_*)^p)\mathbf 1_{E_R^c}
 \leq2 d_X(x,x_*)^p\mathbf 1_{\{d_X(x,x_*)>R\}}
    +2 d_X(x',x_*)^p\mathbf 1_{\{d_X(x',x_*)>R\}}.
\end{equation*}
Here $\mathbf 1_{\dummy}$ is the indicator function. Integrating the right-hand side gives at most $4\tau_{\mathfrak C,p}(R)$, leading to the stated bound on the integral over $E_R^c$. First choose $R$ sufficiently large, then $r$ sufficiently small, and finally $W_p(\mu,\nu)$ sufficiently small; these choices are uniform in $\mu,\nu\in\mathfrak C$.
\end{proof}

\subsection{Dissecting the Assumptions}
\label{ssec:uniform_stability_examples}

The next four examples show that none of the following assumptions in Proposition~\ref{prop:uniform_stability} can be omitted. We repeatedly use the following formula for two-atom probability measures on $\mathbb R$:
\begin{equation}
\label{eq:two_atom_Wp}
 W_p\left((1-a)\delta_0+a\delta_r,
 (1-\widetilde a)\delta_0+\widetilde a\delta_s\right)
 =\left(a(s-r)^p+(\widetilde a-a)s^p\right)^{1/p},
\end{equation}
valid for $0\leq a\leq\widetilde a\leq1$ and $0\leq r\leq s$ by monotone transport on the line. The two terms move the shared mass $a$ from $r$ to $s$ and the excess mass $\widetilde a-a$ from $0$ to $s$, respectively.

\begin{example}[Continuity Is Insufficient in Infinite Dimensions]
\label{ex:likelihood_regularity}
In an infinite-dimensional space, a continuous likelihood need not be uniformly continuous on bounded sets. Let $X=\ell^2$, let $(e_n)_{n\geq1}$ be its canonical orthonormal basis, and define
\begin{equation}\label{eq:likelihood_inf_dim}
     L(x)\coloneqq\frac{1}{2}\left(1+\sum_{n=2}^\infty
 \bigl(1-n\|x-e_n\|_X\bigr)_+\right).
\end{equation}
Here $(t)_+\coloneqq\max\{t,0\}$. For $n\geq2$, the $n$-th summand is supported on the closed ball centered at $e_n$ with radius $n^{-1}$. These balls are pairwise separated by at least $\sqrt2-5/6$ and hence form a locally finite family. Thus $L$ is continuous and $1/2\leq L\leq1$.

For $n\geq2$, set $x_n\coloneqq e_n$ and $y_n\coloneqq(1-n^{-1})e_n$. Both points lie in $B_X(1)$, but $\|x_n-y_n\|_X=n^{-1}\to0$ as $n\to\infty$, whereas $|L(x_n)-L(y_n)|=1/2$. Hence $L$ is not uniformly continuous on bounded sets.

Define
\begin{equation*}
     \mu_n=\frac{1}{2}\delta_0+\frac{1}{2}\delta_{x_n},
 \qquad
 \nu_n=\frac{1}{2}\delta_0+
       \frac{1}{2}\delta_{y_n}.
\end{equation*}
These priors are supported on $B_X(1)$, so their $p$-th moments are uniformly integrable; $Z_{\mu_n}=3/4$, $Z_{\nu_n}=1/2$, and $W_p(\mu_n,\nu_n)^p=(2n^p)^{-1}\to0$ as $n\to\infty$.
Since $L(0)=L(y_n)=1/2$ and $L(x_n)=1$, the posteriors are
\begin{equation*}
    \mathcal B_L(\mu_n)=\frac{1}{3}\delta_0+\frac{2}{3}\delta_{x_n},
 \qquad
 \mathcal B_L(\nu_n)=\frac{1}{2}\delta_0+
       \frac{1}{2}\delta_{y_n}.
\end{equation*}
Both posteriors are supported on the ray $\mathbb R_+e_n$. Applying \eqref{eq:two_atom_Wp} to $W_p(\mathcal B_L(\nu_n),\mathcal B_L(\mu_n))$, with $(a,\widetilde a,r,s)=(1/2,2/3,1-1/n,1)$, and using symmetry gives
\begin{equation*}
     W_p\bigl(\mathcal B_L(\mu_n),\mathcal B_L(\nu_n)\bigr)^p
 =\frac{1}{6}+\frac{1}{2n^p}\longrightarrow\frac{1}{6} \quad\text{as }n\to\infty.
\end{equation*}
Thus, in infinite-dimensional spaces, continuity cannot replace the uniform continuity assumption in Proposition~\ref{prop:uniform_stability}.
\end{example}

\begin{example}[Failure Without a Bounded Likelihood]
\label{ex:unbounded_likelihood}
We consider an unbounded likelihood, which the normalization in \eqref{eq:likelihood_normalization} does not apply. Let $X=\mathbb R$, $L(x)=1+|x|$, and
\begin{equation*}
     \mu=\delta_0,
 \qquad
 \nu_n=(1-n^{-(p+1)})\delta_0+n^{-(p+1)}\delta_n.
\end{equation*}
The likelihood is globally Lipschitz, $Z_\mu=1$, $Z_{\nu_n}=1+n^{-p}$, and the $p$-th moments are uniformly integrable because
\begin{equation*}
     \sup_{n\geq1}\int_{\{|x|>R\}}|x|^p\,\nu_n(\mathrm{d}x)
 =\sup_{n>R}\frac{1}{n}
 \leq\frac{1}{R}\longrightarrow0 \quad\text{as }R\to\infty.
\end{equation*}
Moreover, $W_p(\mu,\nu_n)^p=n^{-1}\to0$ as $n\to\infty$, whereas $\mathcal B_L(\mu)=\delta_0$ and $\mathcal B_L(\nu_n)$ assigns mass $(1+n)n^{-(p+1)}/(1+n^{-p})$ to $n$. Consequently,
\begin{equation*}
 W_p\bigl(\mathcal B_L(\mu),\mathcal B_L(\nu_n)\bigr)^p
 =\frac{1+n}{n(1+n^{-p})}\longrightarrow1
 \quad\text{as }n\to\infty.
\end{equation*}
Thus the factor $L(n)=n+1$ turns vanishing tail mass into an order-one posterior transport cost, so boundedness cannot be omitted.
\end{example}

\begin{example}[Uniform $p$-th Moment Bounds Are Insufficient]
\label{ex:pth_moment_tail}
Let $X=\mathbb R$, $L(x)=(2+\sin(\pi x/2))/3$, $R_n=4n$, and $a_n=R_n^{-p}$. The likelihood is bounded and globally Lipschitz, with $L(0)=L(R_n)=2/3$ and $L(R_n+1)=1$.
Define
\begin{equation*}
 \mu_n=(1-a_n)\delta_0+a_n\delta_{R_n},
 \qquad
 \nu_n=(1-a_n)\delta_0+a_n\delta_{R_n+1}.
\end{equation*}
Their evidences are $Z_{\mu_n}=2/3$ and $Z_{\nu_n}=(2+a_n)/3$, while $m_p(\mu_n)=1$ and $m_p(\nu_n)=(1+R_n^{-1})^p$; hence their $p$-th moments are uniformly bounded. They are not uniformly integrable:
\begin{equation*}
 \lim_{R\to\infty}\sup_{n\geq1}
 \int_{\{|x|>R\}}|x|^p\,\mu_n(\mathrm{d}x)=1.
\end{equation*}
The priors move only mass $a_n$ by unit distance, so $W_p(\mu_n,\nu_n)=R_n^{-1}\to0$ as $n\to\infty$. The Bayesian Update changes the tail mass from $a_n$ to $3a_n/(2+a_n)$. Thus the shared mass $a_n$ moves from $R_n$ to $R_n+1$, while the excess $a_n(1-a_n)/(2+a_n)$ moves from $0$ to $R_n+1$. Formula~\eqref{eq:two_atom_Wp} gives
\begin{equation*}
 W_p\bigl(\mathcal B_L(\mu_n),\mathcal B_L(\nu_n)\bigr)^p
 =a_n
 +\frac{a_n(1-a_n)}{2+a_n}(R_n+1)^p
 \longrightarrow\frac{1}{2}
 \quad\text{as }n\to\infty.
\end{equation*}
This shows that uniform $p$-th moment bounds cannot replace uniform integrability.
\end{example}

\begin{example}[Failure Without a Uniform Evidence Lower Bound]
\label{ex:evidence_collapse}
Let $X=[0,1]$ and let $L(x)=x$. For $\varepsilon\in(0,1)$, set
\begin{equation*}
 \mu_\varepsilon=(1-\varepsilon)\delta_0+\varepsilon\delta_1,
 \qquad
 \nu_\varepsilon=(1-\varepsilon)\delta_0
 +\varepsilon\delta_{1/2}.
\end{equation*}
The likelihood is globally Lipschitz, and the prior $p$-th moments are uniformly integrable because all the priors are supported on $[0,1]$. Their positive evidences, $Z_{\mu_\varepsilon}=\varepsilon$ and $Z_{\nu_\varepsilon}=\varepsilon/2$, have infimum zero. Although $W_p(\mu_\varepsilon,\nu_\varepsilon)=\varepsilon^{1/p}/2\to0$ as $\varepsilon\downarrow0$, the equality $L(0)=0$ means that the Bayesian update assigns no mass to the common atom at $0$, giving $\mathcal B_L(\mu_\varepsilon)=\delta_1$ and $\mathcal B_L(\nu_\varepsilon)=\delta_{1/2}$. Hence the posterior distance remains $1/2$ as $\varepsilon\downarrow0$, so a uniform evidence lower bound is necessary.
\end{example}

\section{H\"older Stability}\label{sec:holder_stability}

We now strengthen the uniform continuity result by deriving an explicit H\"older modulus. A global H\"older bound on the likelihood and a uniform $q$-moment bound on the priors, with $q>p$, give a power-law estimate in $W_p$.

\begin{theorem}[H\"older Stability]
\label{thm:holder_stability}
Let $L:X\to[0,1]$
be globally $\beta$-H\"older with $K\coloneqq\operatorname{Hol}_\beta(L)$ for some $\beta\in(0,1]$. Let $p\geq 1$ and $q>p$. Set $\alpha\coloneqq\min\left\{\beta/p,1-p/q\right\}$.
For any $B\geq 0$ and $\zeta\in(0,1]$, define the admissible prior class 
\begin{equation*}
    \mathfrak C_q(B,\zeta;L)
\coloneqq\left\{\mu\in\mathcal P_q(X)\,\middle|\,
m_q(\mu)\leq B,\,
\mu(L)\geq\zeta\right\}.
\end{equation*}
Then, there exists $C_{\mathrm{stab}}>0$ that depends on  $\beta$, $K$, $p$, $q$, $B$, and $\zeta$ such that
\begin{equation*}
    W_p\bigl(\mathcal B_L(\mu),\mathcal B_L(\nu)\bigr)
\leq C_\mathrm{stab}\, W_p(\mu,\nu)^\alpha\,\qquad\forall \mu,\nu\in\mathfrak C_q(B,\zeta;L).
\end{equation*}
\end{theorem}

\begin{proof}
If $B=0$, the prior class contains at most $\delta_{x_*}$ and the result is immediate. Suppose $B>0$, and let $\pi$ be a $W_p$-optimal coupling of $\mu$ and $\nu$. Since $p\alpha\leq\beta$,
\begin{equation*}
     |L(x)-L(x')|
 \leq\min\{1,K d_X(x,x')^\beta\}
 \leq K^{p\alpha/\beta}d_X(x,x')^{p\alpha}.
\end{equation*}
Because $p\alpha\leq p$, Lyapunov's inequality gives
\begin{equation*}
     \int_{X\times X}|L(x)-L(x')|\,\pi(\mathrm{d}x,\mathrm{d}x')
 \leq K^{p\alpha/\beta}W_p(\mu,\nu)^{p\alpha}.
\end{equation*}
Moreover, H\"older's, Minkowski's, and Lyapunov's inequalities, together
with $p\alpha q/(q-p)\leq p$, give
\begin{align*}
&\int_{X\times X}(d_X(x,x_*)^p+d_X(x',x_*)^p)
 |L(x)-L(x')|\,\pi(\mathrm{d}x,\mathrm{d}x')\\
&\quad\leq
 \left(\int_{X\times X}(d_X(x,x_*)^p+d_X(x',x_*)^p)^{q/p}
 \,\pi(\mathrm{d}x,\mathrm{d}x')\right)^{p/q}\\
&\qquad\times
 \left(\int_{X\times X}|L(x)-L(x')|^{q/(q-p)}
 \,\pi(\mathrm{d}x,\mathrm{d}x')\right)^{1-p/q}\\
&\quad\leq
 2B^{p/q}K^{p\alpha/\beta}W_p(\mu,\nu)^{p\alpha}.
\end{align*}

Finally,
\begin{equation*}
     m_p(\mu)+m_p(\nu)\leq2B^{p/q},
 \qquad
 W_{p}(\mu,\nu)\leq m_p(\mu)^{1/p}+m_p(\nu)^{1/p}\leq2B^{1/q},
\end{equation*}
and hence
\begin{equation*}
      W_{p}(\mu,\nu)^p
 \leq2^{p(1-\alpha)}B^{p(1-\alpha)/q} W_{p}(\mu,\nu)^{p\alpha}.
\end{equation*}
Substitution in the estimate in Lemma~\ref{lem:posterior_coupling_estimate} gives
\begin{equation*}
     W_p\bigl(\mathcal B_L(\mu),\mathcal B_L(\nu)\bigr)^p
 \leq\zeta^{-1}\left(
2^{p(1-\alpha)}B^{p(1-\alpha)/q}
 +2^pB^{p/q}K^{p\alpha/\beta}
 \left(1+\zeta^{-1}\right)
 \right) W_{p}(\mu,\nu)^{p\alpha}.
\end{equation*}
Taking both sides to the power $1/p$ proves the estimate.
\end{proof}

\begin{remark}[The Two Exponent Regimes]
The two branches of $\alpha$ describe distinct limitations of the estimate. The branch $\beta/p$ arises from likelihood reweighting: a prior perturbation of $W_p$ size $h$ can change the posterior mass assigned to a region by order $h^\beta$. If this mass imbalance must be transported over a distance of order one, its contribution to $W_p$ is of order $h^{\beta/p}$. This effect persists for priors supported in a common bounded set. The branch $1-p/q$ accounts for small amounts of prior mass at large distances. A uniform $q$-th moment bound limits their size, but likelihood reweighting can still require transporting some of this mass over those distances. Thus increasing $q$ can improve the exponent up to the limit $\beta/p$. Examples~\ref{ex:escaping_tail} and \ref{ex:holder_sharpness} make these two limitations explicit.
\end{remark}

\begin{remark}[Limits of Moment Assumptions]
For a globally Lipschitz likelihood, the exponent approaches $1$ as $q\to\infty$ when $p=1$, but is at most $1/p$ when $p>1$. Example~\ref{ex:holder_sharpness} shows that bounded support alone does not remove this latter obstruction. Section~\ref{sec:lipschitz_stability} therefore considers Poincar\'e bounds, which control the geometry of the priors, together with bounded oscillation of the potential on their supports.
\end{remark}

\subsection{Likelihood Regularity and Exponent Sharpness}
\label{sec:holder_sharpness}

The first example rules out a uniform power modulus under local Lipschitz continuity; the last two show that both branches of the exponent in Theorem~\ref{thm:holder_stability} are sharp. We write $u_n\sim v_n$ if $u_n/v_n\to1$ as $n\to\infty$, and $u_n\asymp v_n$ if their ratio is bounded above and away from zero as $n\to\infty$; analogous notation is used as $h\downarrow0$.

\begin{example}[Local Lipschitz Continuity Gives No Uniform Power Modulus]
\label{ex:local_no_holder}
A locally Lipschitz likelihood can change by order one across intervals whose widths shrink rapidly at infinity. Let $X=\mathbb R$ and, for $n\geq1$, define
\begin{equation*}
 T_n(x)\coloneqq\bigl(1-2^n|x-n|\bigr)_+,
 \qquad
 L(x)\coloneqq\frac{2+\sum_{n=1}^\infty T_n(x)}{3}.
\end{equation*}
The bump supports $[n-2^{-n},n+2^{-n}]$ are disjoint and locally finite, so $L$ is locally Lipschitz and $2/3\leq L\leq1$. For $r_n=n-2^{-n}$ and $s_n=n$, we have $L(0)=L(r_n)=2/3$ and $L(s_n)=1$, although $s_n-r_n=2^{-n}\to0$ as $n\to\infty$. Hence $L$ is not globally H\"older continuous of any order in $(0,1]$.

With $\varepsilon_n=n^{-q}$, define
\begin{equation*}
 \mu_n=(1-\varepsilon_n)\delta_0+\varepsilon_n\delta_{r_n},
 \qquad
 \nu_n=(1-\varepsilon_n)\delta_0+\varepsilon_n\delta_{s_n}.
\end{equation*}
Since $m_q(\mu_n)=n^{-q}(n-2^{-n})^q\leq1$ and $m_q(\nu_n)=1$, the uniform $q$-th moment bound and $q>p$ imply uniform integrability of the prior $p$-th moments. Moreover, $Z_{\mu_n}=2/3$, $Z_{\nu_n}=(2+\varepsilon_n)/3$, and $W_p(\mu_n,\nu_n)=2^{-n}n^{-q/p}$.

The Bayesian update leaves $\mu_n$ unchanged and increases the tail mass of $\nu_n$ to $3\varepsilon_n/(2+\varepsilon_n)$. Thus the shared mass $\varepsilon_n$ moves from $r_n$ to $s_n$, while the excess $\varepsilon_n(1-\varepsilon_n)/(2+\varepsilon_n)$ moves from $0$ to $s_n$. Formula~\eqref{eq:two_atom_Wp} gives
\begin{equation*}
 W_p\bigl(\mathcal B_L(\mu_n),\mathcal B_L(\nu_n)\bigr)^p
 =\varepsilon_n2^{-np}
 +\frac{\varepsilon_n(1-\varepsilon_n)}{2+\varepsilon_n}n^p
 \sim\frac12n^{p-q}
 \quad\text{as }n\to\infty.
\end{equation*}
Consequently,
\begin{equation*}
 W_p\bigl(\mathcal B_L(\mu_n),\mathcal B_L(\nu_n)\bigr)
 \sim2^{-1/p}n^{1-q/p}\longrightarrow0
 \quad\text{as }n\to\infty.
\end{equation*}
As $n\to\infty$, this decay is polynomial, whereas the prior distance decays exponentially. Indeed, for every $\gamma>0$,
\begin{equation*}
 \frac{W_p(\mathcal B_L(\mu_n),\mathcal B_L(\nu_n))}
      {W_p(\mu_n,\nu_n)^\gamma}
 \asymp 2^{n\gamma}n^{1-q/p+\gamma q/p}\longrightarrow\infty
 \quad\text{as }n\to\infty.
\end{equation*}
Thus no positive power of the prior distance is a uniform modulus along these pairs.
\end{example}

\begin{example}[Sharpness of the Exponent $1-p/q$]
\label{ex:escaping_tail}
We revisit the escaping-tail mechanism in Example~\ref{ex:pth_moment_tail}, now reducing the tail mass from $R_n^{-p}$ to $R_n^{-q}$ to impose a uniform $q$-th moment bound. Let $X=\mathbb R$ and set
\begin{equation*}
 L(x)=\frac{2+\sin(\pi x/2)}{3},
 \qquad
 R_n=4n,
 \qquad
 a_n=R_n^{-q}.
\end{equation*}
Then $L(0)=L(R_n)=2/3$ and $L(R_n+1)=1$. Being bounded and globally Lipschitz, $L$ is globally $\beta$-H\"older.
Define
\begin{equation*}
 \mu_n=(1-a_n)\delta_0+a_n\delta_{R_n},
 \qquad
 \nu_n=(1-a_n)\delta_0+a_n\delta_{R_n+1}.
\end{equation*}
The evidences are $Z_{\mu_n}=2/3$ and $Z_{\nu_n}=(2+a_n)/3$, while $m_q(\mu_n)=1$, $m_q(\nu_n)=(1+R_n^{-1})^q$, and $W_p(\mu_n,\nu_n)=R_n^{-q/p}\to0$ as $n\to\infty$. The Bayesian update changes the tail mass from $a_n$ to $3a_n/(2+a_n)$; the shared mass moves from $R_n$ to $R_n+1$, while the excess $a_n(1-a_n)/(2+a_n)$ moves from $0$ to $R_n+1$. Therefore
\begin{equation*}
 W_p\bigl(\mathcal B_L(\mu_n),\mathcal B_L(\nu_n)\bigr)^p
 =a_n
 +\frac{a_n(1-a_n)}{2+a_n}(R_n+1)^p
 \sim\frac12R_n^{p-q}
 \quad\text{as }n\to\infty.
\end{equation*}
Consequently, $W_p(\mathcal B_L(\mu_n),\mathcal B_L(\nu_n))\sim2^{-1/p}W_p(\mu_n,\nu_n)^{1-p/q}$ as $n\to\infty$. Hence, when $1-p/q\leq\beta/p$, no larger exponent can hold uniformly for these pairs.
\end{example}

\begin{example}[Sharpness of the Exponent $\beta/p$]
\label{ex:holder_sharpness}
As $h\downarrow0$, a displacement of size $h$ changes the likelihood by order $h^\beta$. The Bayesian update turns this change into a posterior mass imbalance of the same order, transported over a distance of order one. Let $X=[0,1]$ and set
\begin{equation*}
 L(x)\coloneqq \frac{1+x^\beta}{2},
 \qquad
 \mu\coloneqq \frac12\delta_0+\frac12\delta_1,
 \qquad
 \nu_h\coloneqq \frac12\delta_h+\frac12\delta_1,
 \quad 0<h<1.
\end{equation*}
Since $|x^\beta-y^\beta|\leq|x-y|^\beta$ on $[0,1]$, the likelihood is globally $\beta$-H\"older with constant $1/2$. All these priors are supported on $[0,1]$, so their moments of every order are uniformly bounded. Their evidences are $Z_\mu=3/4$ and $Z_{\nu_h}=(3+h^\beta)/4\geq3/4$, and $W_p(\mu,\nu_h)=2^{-1/p}h\to0$ as $h\downarrow0$. The posterior mass at the lower atom increases from $1/3$ to $(1+h^\beta)/(3+h^\beta)$, an excess of $\frac{2h^\beta}{3(3+h^\beta)}$.
Under monotone transport, mass $1/3$ moves from $0$ to $h$, while this excess mass moves from $1$ to $h$. Hence
\begin{equation*}
 W_p\!\left(
 \mathcal B_L(\mu),
 \mathcal B_L(\nu_h)
 \right)^p
 =\frac{h^p}{3}
 +\frac{2h^\beta}{3(3+h^\beta)}(1-h)^p
 \asymp h^\beta
 \quad\text{as }h\downarrow0.
\end{equation*}
As $h\downarrow0$, the second term is asymptotic to $(2/9)h^\beta$, while $h^p=O(h^\beta)$ since $p\geq1\geq\beta$. Therefore $W_p(\mathcal B_L(\mu),\mathcal B_L(\nu_h))\asymp W_p(\mu,\nu_h)^{\beta/p}$ as $h\downarrow0$.
When $\beta/p\leq1-p/q$, the exponent in Theorem~\ref{thm:holder_stability} is $\beta/p$, and no larger exponent can hold uniformly for these pairs, even though all the priors are supported on $[0,1]$.
\end{example}

\section{Lipschitz Stability}\label{sec:lipschitz_stability}

In this section, we present our third contribution on $W_p$-Lipschitz stability with $p>1$. We replace the moment bounds used in Section~\ref{sec:holder_stability} by uniform bounds on the Poincar\'e constant and the variation of the potential under each prior. After introducing these conditions, we prove the stability estimate in Theorem~\ref{thm:lipschitz_stability} and discuss its applications.

\subsection{Poincar\'e Constant}

For a bounded Lipschitz function $f:X\to\mathbb R$, denote its pointwise local Lipschitz constant by
\begin{equation*}
     |\mathrm D f|(x)
 \coloneqq\limsup_{\substack{x'\to x\\x'\neq x}}
 \frac{|f(x')-f(x)|}{d_X(x',x)},
\end{equation*}
with value zero at isolated points. When $X$ is a Banach space and $f$ is Fr\'echet differentiable at $x$, $|\mathrm D f|(x)=\|Df(x)\|_{X^*}$, where $X^*$ is the topological dual of $X$. For $r>1$ and $\mu\in\mathcal P(X)$, define the $L^r$-Poincar\'e constant by
\begin{equation}\label{eq:lp_poincare}
\begin{aligned}
 C_{\mathrm{P},r}(\mu)\coloneqq\inf\Big\{C\geq0\Bigm|&
 \inf_{a\in\mathbb R}\int_X|f(x)-a|^r\,\mu(\mathrm{d}x)
 \leq C\int_X|\mathrm Df|^r(x)\,\mu(\mathrm{d}x)\\
 &\text{for every bounded Lipschitz }f:X\to\mathbb R\Big\},
\end{aligned}
\end{equation}
where the infimum of the empty set is $+\infty$.

For $p>1$, uniform control of this constant with $r=p'\coloneqq p/(p-1)$ will be central to establishing $W_p$-Lipschitz stability. Poincar\'e constants control the variation of functions through their pointwise local Lipschitz constants. On normed spaces, they are translation invariant, unlike moment bounds. For example, the uniform distributions on $[n,n+1]$ have the same finite $L^r$-Poincar\'e constant for all $n\in\mathbb N$, while their moments of every positive order diverge as $n\to\infty$. Conversely, a probability measure whose support consists of two distinct points has bounded moments of every order but an infinite Poincar\'e constant, since a bounded Lipschitz function can take different values at the two points while being constant in a neighborhood of each.

\subsection[Essential Oscillation]{$\mu$-Essential Oscillation}

For $\mu\in\mathcal P(X)$ and a real-valued measurable function $f$, define its $\mu$-essential oscillation by
\begin{equation*}
 \operatorname{osc}_{\mu}(f)
 \coloneqq\operatorname*{ess\,sup}_{\mu}f
   -\operatorname*{ess\,inf}_{\mu}f\in[0,\infty].
\end{equation*}
Here the essential supremum and essential infimum are taken with respect to $\mu$. 

We impose a uniform bound on the $\mu$-essential oscillation of the \emph{potential} $\Phi:X\to[0,\infty)$, which defines the likelihood through $L\coloneqq e^{-\Phi}$. This bound controls likelihood ratios on the part of $X$ seen by each prior; in particular,
\begin{equation*}
 \operatorname{osc}_{\mu}(\Phi)\leq M
 \quad\Longrightarrow\quad
 e^{-M}\leq\frac{L(x)}{Z_\mu}\leq e^M
 \quad\text{for }\mu\text{-almost every }x\in X.
\end{equation*}
Like the posterior, essential oscillation is unchanged when a constant is added to the potential. It therefore measures the relative variation of the likelihood within each prior. The following example distinguishes an oscillation bound under each prior from a bound on the potential over the whole parameter space.

\begin{example}[Interval Priors with Additive Laplace Noise]
\label{ex:laplace_oscillation}
Let $X=\mathbb R$ and consider $y=x+\eta$, where $\eta\sim\operatorname{Laplace}(0,b)$ with $b>0$ is independent of $x$. For an observation $y\in\mathbb R$, the normalized potential is $\Phi(x;y)=|y-x|/b$. Fix $D>0$ and let $\mu_a$ be the uniform distribution on $[a,a+D]$, for $a\in\mathbb R$. The reverse triangle inequality gives $\operatorname{osc}_{\mu_a}(\Phi(\dummy;y))\leq D/b$ for every $a\in\mathbb R$. The bound is uniform over this prior family, whose supports cover $\mathbb R$, even though $\Phi$ is unbounded on $\mathbb R$.
\end{example}

\subsection{Main Result}

The next theorem gives $W_p$-Lipschitz stability for $p>1$ under a globally Lipschitz potential and uniform Poincar\'e and essential-oscillation bounds on the prior class.

\begin{theorem}[Lipschitz Stability]
\label{thm:lipschitz_stability}
Let $\Phi:X\to[0,\infty)$ be globally Lipschitz with $K\coloneqq\operatorname{Lip}(\Phi)$ and $\inf_{x\in X}\Phi(x)=0$, and set $L\coloneqq e^{-\Phi}$. Let $p>1$ and $p'\coloneqq p/(p-1)$. For $B,M\geq0$, define
\begin{equation}\label{eq:lipschitz_prior_class}
 \mathfrak C_p(B,M;\Phi)
 \coloneqq\left\{\mu\in\mathcal P_p(X)\,\middle|\,
 C_{\mathrm{P},p'}(\mu)\leq B,\,
 \operatorname{osc}_{\mu}(\Phi)\leq M\right\}.
\end{equation}
Then, for all $\mu,\nu\in\mathfrak C_p(B,M;\Phi)$,
\begin{equation}\label{eq:lipschitz_stability}
 W_p\bigl(\mathcal B_L(\mu),\mathcal B_L(\nu)\bigr)
 \leq\left(e^{M/p}+\gamma_p K B^{1/p'}e^M\right)W_p(\mu,\nu),
\end{equation}
where $\gamma_p<2$ and $\gamma_p=1$ when $p=2$.
\end{theorem}

\begin{proof}
We first outline the proof strategy. Fix $\mu,\nu\in\mathfrak C_p(B,M;\Phi)$ and a $W_p$-optimal coupling $\pi\in\Pi(\mu,\nu)$. For $t\in[0,1]$, define
\begin{gather*}
 F_t(x,x')\coloneqq e^{-t\Phi(x)-(1-t)\Phi(x')},
 \qquad
 Z_t\coloneqq\int_{X\times X}F_t(x,x')\,\pi(\mathrm{d}x,\mathrm{d}x'),\\
 \pi_t(\mathrm{d}x,\mathrm{d}x')\coloneqq
 Z_t^{-1}F_t(x,x')\,\pi(\mathrm{d}x,\mathrm{d}x'),
\end{gather*}
and let $\rho_t^{(1)}$ and $\rho_t^{(2)}$ denote the first and second marginals of $\pi_t$, respectively. Since $F_1(x,x')=L(x)$ and $F_0(x,x')=L(x')$,
\begin{equation*}
 \rho_1^{(1)}=\mathcal B_L(\mu),
 \qquad
 \rho_0^{(2)}=\mathcal B_L(\nu).
\end{equation*}
The triangle inequality therefore gives the decomposition
\begin{equation}
\label{eq:lipschitz_triangle_split}
 W_p\bigl(\mathcal B_L(\mu),\mathcal B_L(\nu)\bigr)
 \leq W_p\bigl(\rho_1^{(1)},\rho_1^{(2)}\bigr)
     +W_p\bigl(\rho_1^{(2)},\rho_0^{(2)}\bigr).
\end{equation}
We control the two terms using $W_p(\mu,\nu)$ separately, relying on the fact that $Z_t^{-1}F_t$ is uniformly bounded from above and below $\pi$-almost everywhere due to the bounded essential oscillation of $\Phi$. While the estimate for the first term is straightforward, the second term requires an estimate along the interpolating curve. We use
Lemma~\ref{lem:dual_wasserstein_speed}, which gives
\begin{equation}
\label{eq:lipschitz_second_marginal_target}
 W_p\bigl(\rho_s^{(2)},\rho_t^{(2)}\bigr)
 \leq (t-s) C^{(2)}_{\rm stab} W_{p}(\mu,\nu),
 \qquad 0\leq s\leq t\leq1,
\end{equation}
provided the following three conditions hold for some $C^{(2)}_{\rm stab}\geq 0$:
\begin{enumerate}[label=(\roman*)]
    \item $\rho_t^{(2)}\in\mathcal P_p(X)$ for every $t\in[0,1]$;
    \item there exist $\vartheta\in\mathcal P(X)$ and $C_{\rm dom}>0$ such that $\rho_t^{(2)}\leq C_{\rm dom}\vartheta$ for every $t\in[0,1]$;
    \item for every bounded Lipschitz function $f:X\to\mathbb R$, the map $t\mapsto\rho_t^{(2)}(f)$ is absolutely continuous and satisfies
\begin{equation}
\label{eq:lipschitz_derivative_target}
 \left|\frac{\mathrm d}{\mathrm dt}\rho_t^{(2)}(f)\right|
 \leq C^{(2)}_{\rm stab} W_{p}(\mu,\nu)
 \bigl\||\mathrm Df|\bigr\|_{L^{p'}_{\rho_t^{(2)}}(X)}
 \quad\text{for almost every }t\in(0,1).
\end{equation}
\end{enumerate}
We first estimate the first term in \eqref{eq:lipschitz_triangle_split} and then verify conditions (i)--(iii) to control the second term.

For $\lambda\in\{\mu,\nu\}$, let $a_\lambda\coloneqq\operatorname*{ess\,inf}_\lambda\Phi$ and $b_\lambda\coloneqq\operatorname*{ess\,sup}_\lambda\Phi$. These quantities are finite, and $b_\lambda-a_\lambda\leq M$. For $\pi$-almost every $(x,x')$, the value $F_t(x,x')$ lies between $e^{-tb_\mu-(1-t)b_\nu}$ and $e^{-ta_\mu-(1-t)a_\nu}$. The same bounds hold for $Z_t$. Their ratio is at most $e^M$, and since $\mathrm d\pi_t/\mathrm d\pi=F_t/Z_t$, it follows that
\begin{equation}\label{eq:rn_bound}
 e^{-M}\leq Z_t^{-1}F_t(x,x')\leq e^M
 \quad\text{for }\pi\text{-almost every }(x,x')\in X\times X.
\end{equation}
At $t=1$, the coupling $\pi_1$ has marginals $\rho_1^{(1)}$ and $\rho_1^{(2)}$. Hence
\begin{equation}\label{eq:lipschitz_coupling_term}
 W_p\bigl(\rho_1^{(1)},\rho_1^{(2)}\bigr)
 \leq\left(\int_{X\times X}d_X(x,x')^p\,\pi_1(\mathrm{d}x,\mathrm{d}x')\right)^{1/p}
 \leq e^{M/p}W_p(\mu,\nu).
\end{equation}
We thus have the estimate for the first term in \eqref{eq:lipschitz_triangle_split}.

Projecting the upper bound in \eqref{eq:rn_bound} onto the two coordinates gives $\rho_t^{(1)}\leq e^M\mu$ and $\rho_t^{(2)}\leq e^M\nu$ for $0\leq t\leq1$. In particular,
\begin{equation}
\label{eq:lipschitz_second_marginal_moment}
 \sup_{t\in[0,1]}\int_X d_X(x',x_*)^p\,\rho_t^{(2)}(\mathrm{d}x')
 \leq e^M m_p(\nu)<\infty.
\end{equation}
Thus $\rho_t^{(2)}\in\mathcal P_p(X)$ for every $t\in[0,1]$, which verifies (i). The projected density bound verifies (ii) with $\vartheta=\nu$ and $C_{\rm dom}=e^M$.

Disintegrate $\pi(\mathrm{d}x,\mathrm{d}x')=\pi_{x'}(\mathrm{d}x)\nu(\mathrm{d}x')$ and set
\begin{equation*}
 h_t(x')\coloneqq e^{-(1-t)\Phi(x')}
 \int_X e^{-t\Phi(x)}\,\pi_{x'}(\mathrm{d}x).
\end{equation*}
For $\nu$-almost every $x'\in X$, we have $e^{-tb_\mu-(1-t)b_\nu}\leq h_t(x')
 \leq e^{-ta_\mu-(1-t)a_\nu}$. Thus $q_t\coloneqq\mathrm d\rho_t^{(2)}/\mathrm d\nu=h_t/Z_t$ has positive essential infimum $\underline q_t$ and finite essential supremum $\overline q_t$, with $\overline q_t/\underline q_t\leq e^M$. For $r=p'$ and every bounded Lipschitz function $f:X\to\mathbb R$, we obtain
\begin{align*}
 \inf_{a\in\mathbb R}\int_X|f(x')-a|^r\,\rho_t^{(2)}(\mathrm{d}x')
 &\leq\overline q_t\inf_{a\in\mathbb R}
       \int_X|f(x')-a|^r\,\nu(\mathrm{d}x')\\
 &\leq\overline q_t C_{\mathrm{P},r}(\nu)
       \int_X|\mathrm Df|^r(x')\,\nu(\mathrm{d}x')\\
 &\leq\frac{\overline q_t}{\underline q_t}C_{\mathrm{P},r}(\nu)
       \int_X|\mathrm Df|^r(x')\,\rho_t^{(2)}(\mathrm{d}x').
\end{align*}
Consequently,
\begin{equation}\label{eq:rho_t_poincare}
 C_{\mathrm{P},p'}(\rho_t^{(2)})
 \leq e^M C_{\mathrm{P},p'}(\nu)
 \leq e^M B.
\end{equation}
Since $0<F_t\leq1$,
\begin{equation*}
 \partial_tF_t(x,x')
 =\bigl(\Phi(x')-\Phi(x)\bigr)F_t(x,x'),
 \qquad
 |\partial_tF_t(x,x')|
 \leq|\Phi(x)-\Phi(x')|
 \leq K d_X(x,x'),
\end{equation*}
and the last expression is $\pi$-integrable because $p>1$ and $\pi$ has finite $p$-transport cost. Dominated convergence shows that $Z_t$ is continuous. Since $Z_t>0$ for every $t\in[0,1]$, compactness gives $\inf_{t\in[0,1]}Z_t>0$. For a bounded Lipschitz function $f:X\to\mathbb R$, set
\begin{equation*}
 N_t(f)\coloneqq\int_{X\times X}f(x')F_t(x,x')\,\pi(\mathrm{d}x,\mathrm{d}x'),
 \qquad
 \rho_t^{(2)}(f)=\frac{N_t(f)}{Z_t}.
\end{equation*}
This bound, and its product with $\|f\|_\infty$ for $N_t(f)$, provide integrable dominating functions that justify differentiation under the two integrals. Dominated convergence applied to the resulting derivatives shows that $N_t(f)$ and $Z_t$ are continuously differentiable. Since $Z_t$ is bounded away from zero, the map $t\mapsto\rho_t^{(2)}(f)$ is continuously differentiable and therefore absolutely continuous. This verifies the absolute-continuity part of (iii).

The quotient rule gives
\begin{align*}
 \frac{\mathrm d}{\mathrm dt}\rho_t^{(2)}(f)
 =\int_{X\times X}
 \bigl(f(x')-\rho_t^{(2)}(f)\bigr)
 \bigl(\Phi(x')-\Phi(x)\bigr)\,\pi_t(\mathrm{d}x,\mathrm{d}x').
\end{align*}
By \eqref{eq:rn_bound} and the Lipschitz bound on $\Phi$,
\begin{equation*}
 \left(\int_{X\times X}|\Phi(x)-\Phi(x')|^p\,\pi_t(\mathrm{d}x,\mathrm{d}x')\right)^{1/p}
 \leq K e^{M/p}W_p(\mu,\nu).
\end{equation*}
Applying H\"older's inequality, Lemma~\ref{lem:sharp_centering}, and the Poincar\'e bound \eqref{eq:rho_t_poincare} in sequence gives
\begin{align*}
 \left|\frac{\mathrm d}{\mathrm dt}\rho_t^{(2)}(f)\right|
 &\leq\|f-\rho_t^{(2)}(f)\|_{L^{p'}_{\rho_t^{(2)}}(X)}
       K e^{M/p}W_p(\mu,\nu)\\
 &\leq\gamma_p\inf_{a\in\mathbb R}\|f-a\|_{L^{p'}_{\rho_t^{(2)}}(X)}
       K e^{M/p}W_p(\mu,\nu)\\
 &\leq\gamma_p C_{\mathrm{P},p'}(\rho_t^{(2)})^{1/p'}
       \bigl\||\mathrm Df|\bigr\|_{L^{p'}_{\rho_t^{(2)}}(X)}
       K e^{M/p}W_p(\mu,\nu)\\
 &\leq\gamma_p K B^{1/p'}e^M W_p(\mu,\nu)
       \bigl\||\mathrm Df|\bigr\|_{L^{p'}_{\rho_t^{(2)}}(X)}.
\end{align*}
The final inequality uses \eqref{eq:rho_t_poincare} together with $e^{M/p}(e^M)^{1/p'}=e^M$, which follows from $1/p+1/p'=1$. Thus \eqref{eq:lipschitz_derivative_target} holds with $C^{(2)}_{\rm stab}=\gamma_p K B^{1/p'}e^M$, completing the verification of (iii).

All three conditions of Lemma~\ref{lem:dual_wasserstein_speed} are now satisfied, so \eqref{eq:lipschitz_second_marginal_target} follows. Taking $s=0$ and $t=1$ in that estimate and combining it with \eqref{eq:lipschitz_triangle_split} and \eqref{eq:lipschitz_coupling_term} yields \eqref{eq:lipschitz_stability}.
\end{proof}

\begin{remark}[Verification by Pushforward]
\label{rem:pushforward_verification}
Let $r>1$, let $\lambda\in\mathcal P(X)$ have a finite $L^r$-Poincar\'e constant, and let $T:X\to X$ be globally Lipschitz. The chain rule for pointwise local Lipschitz constants gives
\begin{equation*}
 C_{\mathrm{P},r}(T_\#\lambda)
 \leq\operatorname{Lip}(T)^r C_{\mathrm{P},r}(\lambda),
 \qquad
 \operatorname{osc}_{T_\#\lambda}(\Phi)
 =\operatorname{osc}_{\lambda}(\Phi\circ T),
\end{equation*}
where $T_{\#}\lambda$ denotes the pushforward of $\lambda$ under $T$. Pushforward representations are used in transport-based Bayesian computation and generative modeling \cite{elmoselhy2012bayesian,marzouk2017sampling,papamakarios2021normalizing,cao2026lazydino}. When such representations are used as priors, the conditions in \eqref{eq:lipschitz_prior_class} can be checked on a shared reference measure $\lambda\in\mathcal P_p(X)$: a finite $C_{\mathrm{P},p'}(\lambda)$, together with uniform bounds on $\operatorname{Lip}(T)$ and $\operatorname{osc}_{\lambda}(\Phi\circ T)$, suffices. For two such maps $T,S$, the coupling $(T,S)_\#\lambda$ gives
\begin{equation*}
 W_p(T_\#\lambda,S_\#\lambda)
 \leq\left(\int_X d_X(T(x),S(x))^p\,\lambda(\mathrm{d}x)\right)^{1/p}.
\end{equation*}
Theorem~\ref{thm:lipschitz_stability} then bounds the posterior error by the $L^p_{\lambda}$ norm of the pointwise distance between the maps.
\end{remark}

\subsection{Gaussian Priors and Additive Gaussian Noise}

In this subsection, we consider a commonly used setting where the prior is Gaussian and the observational noise is additive and Gaussian. Let $X$ be a separable Hilbert space of parameters, $Y=\mathbb{R}^d$ be the Euclidean space of observations, and $G:X\to Y$ be the forward model. Consider the observation model
\begin{equation}\label{eq:gaussian_likelihood}
 y=G(x)+\eta,
 \qquad \eta\sim\mathcal N(0,\Gamma),
 \qquad \Phi(x;y)=\tfrac12|y-G(x)|_\Gamma^2,
\end{equation}
where $\eta$ is independent of $x$, $\Gamma$ is positive definite, and $|z|_\Gamma\coloneqq|\Gamma^{-1/2}z|$ with $|\dummy|$ the Euclidean norm. Let $\|\dummy\|_{\mathrm{op}}$ be the operator norm. The following corollary provides a Lipschitz stability estimate for this setting.

\begin{corollary}[Gaussian Priors and Additive Gaussian Noise]
\label{cor:gaussian_prior_stability}
Suppose that $G:X\to Y$ is globally Lipschitz and $\sup_{x\in X}|G(x)|\leq R$ for some $R\geq0$. Fix $y\in Y$, define $\Phi(\dummy;y)$ by \eqref{eq:gaussian_likelihood}, and let $L\propto e^{-\Phi(\dummy;y)}$ be normalized as in \eqref{eq:likelihood_normalization}. For any $B\geq0$, there exists $C_{\mathrm{stab}}>0$, depending only on $B$, $R$, $\operatorname{Lip}(G)$, $|y|$, and $\|\Gamma^{-1}\|_{\mathrm{op}}$, such that
\begin{equation*}
 W_2\bigl(\mathcal B_L(\mu),\mathcal B_L(\nu)\bigr)
 \leq C_{\mathrm{stab}}W_2(\mu,\nu)
\end{equation*}
for all Gaussian priors $\mu=\mathcal N(m,C)$ and $\nu=\mathcal N(\widetilde m,\widetilde C)$ on $X$ with $\|C\|_{\mathrm{op}},\|\widetilde C\|_{\mathrm{op}}\leq B$.
\end{corollary}

\begin{proof}
The potential $\Phi(\dummy;y)$ is bounded and globally Lipschitz, with bounds depending only on $R$, $\operatorname{Lip}(G)$, $|y|$, and $\|\Gamma^{-1}\|_{\mathrm{op}}$. Gaussian measures have finite second moments and satisfy $C_{\mathrm{P},2}(\mathcal N(m,C))\leq\|C\|_{\mathrm{op}}$ by the Gaussian Poincar\'e inequality \cite[Theorem~5.5.1]{bogachev1998gaussian}. Applying Theorem~\ref{thm:lipschitz_stability} with $p=2$ to $\Phi(\dummy;y)-\inf_{x\in X}\Phi(x;y)$ proves the result.
\end{proof}

The prior distance measures perturbations of both the mean and covariance through the formula \cite[Theorem~3.5]{gelbrich1990formula}
\begin{equation*}
 W_2(\mu,\nu)^2
 =\|m-\widetilde m\|_X^2
 +\operatorname{tr}(C)+\operatorname{tr}(\widetilde C)
 -2\operatorname{tr}\bigl((C^{1/2}\widetilde C C^{1/2})^{1/2}\bigr),
\end{equation*}
where $\operatorname{tr}$ denotes the trace. The posterior bound remains valid when these prior perturbations produce mutually singular Gaussian measures. In infinite dimensions, this may occur for a mean shift outside the Cameron--Martin space, or for covariance changes violating the Feldman--H\'ajek equivalence conditions; see \cite[Section~2.7]{sullivan2015introduction}.

\section{A Numerical Example: Two-Layer Darcy Flow}
\label{sec:numerical_example}

In this section, we numerically study the stability of the prior-to-posterior map for a two-layer Darcy flow inverse problem. The two-dimensional model with explicit forward evaluation isolates how prior geometry affects the stability rate. We introduce the observation model in Subsection~\ref{ssec:darcy_observation} and consider Gaussian prior perturbations in Subsection~\ref{ssec:darcy_gaussian}. For Gaussian priors, the results show linear prior-to-posterior sensitivity for small perturbations and illustrate the corresponding bounds on errors in posterior QoI statistics. In Subsection~\ref{ssec:darcy_atomic}, we show a square-root sensitivity for the prior-to-posterior map when using two-point priors, illustrating the different behavior permitted by the H\"older and Lipschitz estimates.

\subsection{Observation Model}
\label{ssec:darcy_observation}

The forward model below is a two-layer specialization of the one-dimensional diffusion model in \cite[Section~3.3]{stuart2010inverse}, after normalizing the boundary pressures. Let $X=\mathbb R^2$ with the Euclidean norm. For $x=(x_1,x_2)\in X$, define the permeability as
\begin{equation*}
 a(x,s)\coloneqq
 \begin{cases}
 \kappa(x_1),&0<s<1/2,\\
 \kappa(x_2),&1/2<s<1,
 \end{cases}
 \qquad
 \kappa(z)\coloneqq\frac32+\frac12\tanh z.
\end{equation*}
The pressure $p_x$ is the weak solution of
\begin{equation*}
 -\frac{\mathrm d}{\mathrm ds}\left(a(x,s)\frac{\mathrm dp_x(s)}{\mathrm ds}\right)=0
 \quad\text{in }(0,1),
 \qquad p_x(0)=0,\quad p_x(1)=1.
\end{equation*}
Continuity of pressure and flux at $s=1/2$ gives the forward map
\begin{equation}\label{eq:darcy_forward}
 G(x)\coloneqq p_x(1/2)
 =\frac{\kappa(x_2)}{\kappa(x_1)+\kappa(x_2)},
 \qquad \frac13<G<\frac23,
 \qquad \operatorname{Lip}(G)\leq\frac{\sqrt2}{8}.
\end{equation}
We take additive Gaussian noise, independent of $x$, with
\begin{equation}\label{eq:darcy_likelihood}
 y=G(x)+\eta,\qquad \eta\sim\mathcal N(0,\sigma^2),\qquad
 \Phi(x)\coloneqq\frac{(y-G(x))^2}{2\sigma^2},\qquad L\coloneqq e^{-\Phi}.
\end{equation}
We set $y=0.55$ and $\sigma=0.15$.

\subsection{Gaussian Prior Perturbations}
\label{ssec:darcy_gaussian}

We perturb the direction of a rank-one Gaussian prior by setting
\begin{equation}\label{eq:darcy_gaussian_priors}
 v\coloneqq\frac{(1,-1)}{\sqrt2},\qquad
 w\coloneqq\frac{(1,1)}{\sqrt2},\qquad
 v_h\coloneqq v+hw,\qquad
 \mu_h\coloneqq\mathcal N(0,v_hv_h^\top),\qquad 0\leq h\leq\frac12.
\end{equation}
Coupling $\xi v$ and $\xi v_h$ with the same $\xi\sim\mathcal N(0,1)$ is optimal, giving
\begin{equation}\label{eq:darcy_prior_error}
 W_2(\mu_0,\mu_h)=h.
\end{equation}
For $h>0$, the priors $\mu_0$ and $\mu_h$ are mutually singular because they are supported on distinct lines. Positivity of $L$ preserves this mutual singularity for the posteriors $\rho_h\coloneqq\mathcal B_L(\mu_h)$. Since $C_{\mathrm P,2}(\mu_h)\leq1+h^2\leq5/4$, Theorem~\ref{thm:lipschitz_stability} gives
\begin{equation}\label{eq:darcy_posterior_bound}
 W_2(\rho_0,\rho_h)\leq C_{\mathrm{stab}}h,
 \qquad C_{\mathrm{stab}}\approx7.0867.
\end{equation}
Figure~\ref{fig:darcy_comparison} (left) shows the posterior distance for twelve values of $h$ between $10^{-4}$ and $0.5$; Appendix~\ref{app:darcy_details} gives the quadrature method for the distance computation and the calculation of $C_{\mathrm{stab}}$. For $h\leq10^{-2}$, the fitted log--log slope is approximately $1$ and $W_2(\rho_0,\rho_h)/h\approx0.9153$. Thus \eqref{eq:darcy_posterior_bound} captures the observed linear rate, but its constant is conservative.

For the first-layer permeability $Q(x)\coloneqq\kappa(x_1)$, define the errors in the posterior mean and standard deviation by
\begin{equation*}
 e_{\mathrm{mean}}(h)\coloneqq|\rho_h(Q)-\rho_0(Q)|,
 \qquad
 e_{\mathrm{sd}}(h)\coloneqq|\operatorname{sd}_{\rho_h}(Q)-\operatorname{sd}_{\rho_0}(Q)|.
\end{equation*}
Since $\operatorname{Lip}(Q)=1/2$, \eqref{eq:mean_bound_wasserstein} and \eqref{eq:sd_bound_wasserstein} imply
\begin{equation}\label{eq:darcy_qoi_bound}
 \max\{e_{\mathrm{mean}}(h),e_{\mathrm{sd}}(h)\}
 \leq\frac12W_2(\rho_0,\rho_h)
 \leq\frac12C_{\mathrm{stab}}h.
\end{equation}
In Figure~\ref{fig:darcy_comparison} (middle), the error in the standard deviation exceeds the error in the mean for every tested $h$, and both lie below the bound in \eqref{eq:darcy_qoi_bound}.

\subsection{Comparison with Two-Point Priors}
\label{ssec:darcy_atomic}

For the same likelihood, consider
\begin{equation*}
 \nu_0\coloneqq\frac12\delta_0+\frac12\delta_v,
 \qquad
 \nu_h\coloneqq\frac12\delta_{hv}+\frac12\delta_v,
 \qquad 0<h\leq\frac12.
\end{equation*}
These priors have uniformly bounded moments of every order and a common positive evidence lower bound, but infinite Poincar\'e constants. The prior distance is $W_2(\nu_0,\nu_h)=h/\sqrt2$, whereas the posterior distance is of order $h^{1/2}$ as $h\downarrow0$; the formula is given in Appendix~\ref{app:darcy_atomic}.  This attains the exponent $1/2$ in our H\"older stability result in Theorem~\ref{thm:holder_stability}, with $\beta=1$, $p=2$, and $q=4$. Figure~\ref{fig:darcy_comparison} (right) contrasts this square-root rate with the linear posterior error for the Gaussian prior family, consistent with our Lipschitz stability result in Theorem~\ref{thm:lipschitz_stability}.

\begin{figure}[H]
 \centering
  
    \addtolength{\tabcolsep}{-5pt}
 \begin{tabular}{c c c}
  \includegraphics[width=0.32\linewidth]{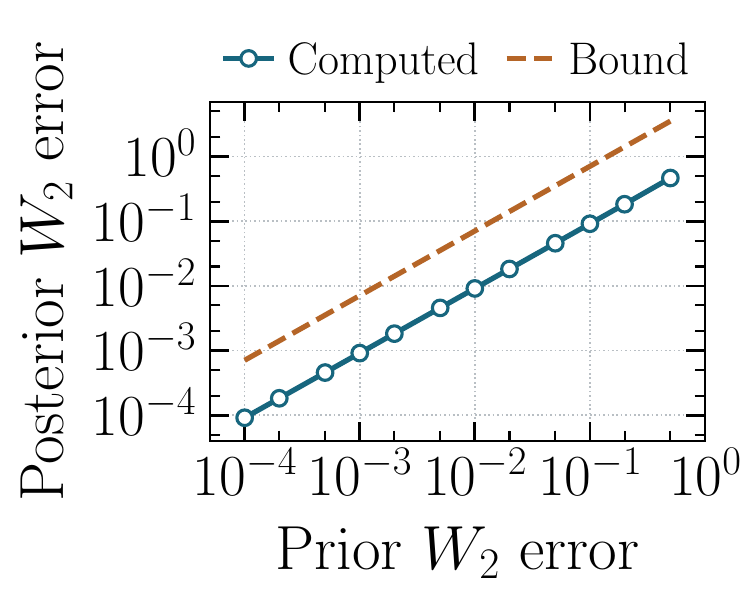} & \includegraphics[width=0.32\linewidth]{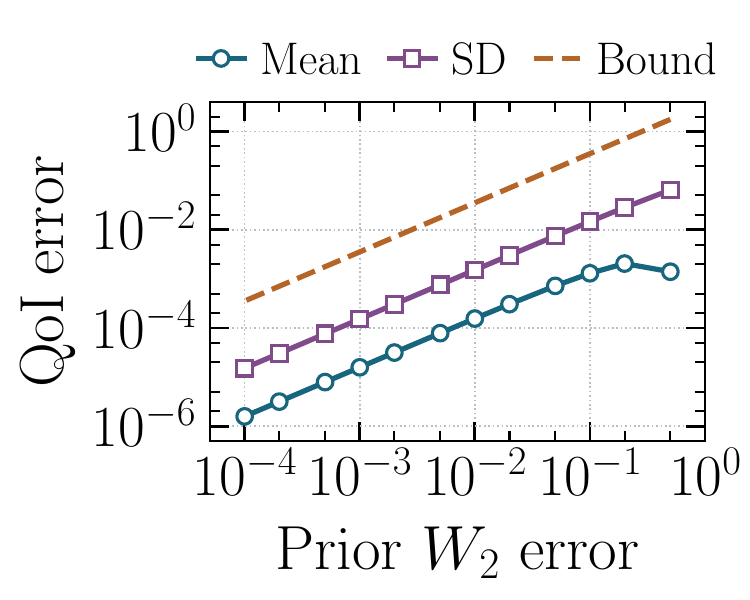} & \includegraphics[width=0.32\linewidth]{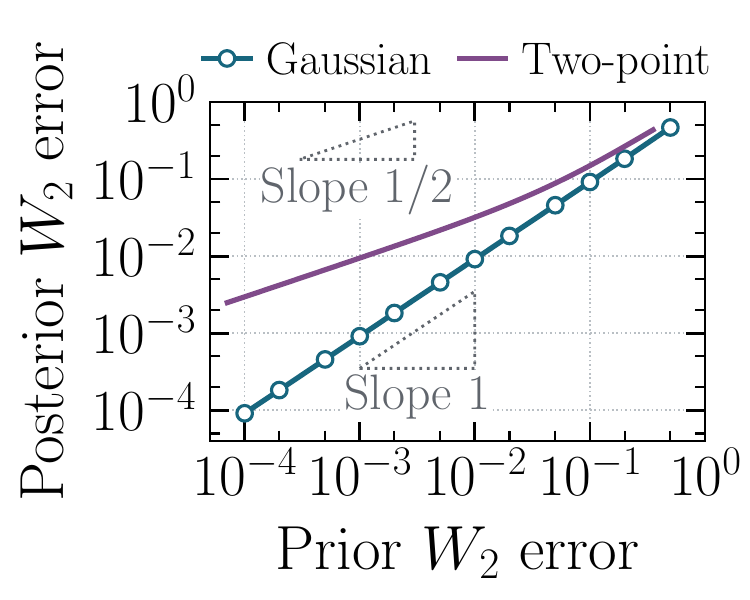}
 \end{tabular}
  \addtolength{\tabcolsep}{5pt}

 \caption{Sensitivity to prior perturbations for the Darcy flow inverse problem. \emph{Left:} Posterior $W_2$ error for Gaussian priors and the bound $C_{\mathrm{stab}}h$ in \eqref{eq:darcy_posterior_bound}. \emph{Middle:} Errors in the posterior mean and standard deviation (SD) of $Q(x)=\kappa(x_1)$ and their shared bound $C_{\mathrm{stab}}h/2$ in \eqref{eq:darcy_qoi_bound}. Dashed curves denote these bounds, with $C_{\mathrm{stab}}\approx7.0867$. \emph{Right:} Under the same likelihood, the posterior error is of order $h$ for Gaussian priors and $h^{1/2}$ for two-point priors as $h\downarrow0$. The horizontal axis shows the prior $W_2$ error, equal to $h$ for Gaussian priors and $h/\sqrt2$ for two-point priors; dotted triangles indicate log--log slopes $1$ and $1/2$.}
 \label{fig:darcy_comparison}
\end{figure}
\FloatBarrier

\section{Conclusion}\label{sec:conclusion}

We established uniform, Hölder, and Lipschitz stability of the prior-to-posterior map in the same Wasserstein metric $W_p$ under explicit assumptions on the likelihood and prior class. The Hölder exponent is sharp under the stated assumptions. These results control posterior approximation errors and, for $p\geq2$, errors in the means and standard deviations of Lipschitz QoIs. For bounded globally Lipschitz forward maps with additive Gaussian noise, the $W_2$ bound applies to Gaussian priors with uniformly bounded covariance operator norms, including mean and covariance perturbations that produce mutually singular measures.

A key takeaway of our stability analysis is that likelihood regularity and moment bounds alone do not ensure Lipschitz stability for $p>1$, even when the priors are supported in a common bounded set. Our Lipschitz result uses additional information about the geometry of the priors, expressed through their $L^{p/(p-1)}$-Poincaré constants. Uniform bounds on these constants, together with a globally Lipschitz potential and uniform bounds on its essential oscillation under the priors, provide sufficient conditions for Lipschitz stability. The Darcy example illustrates this distinction: posterior $W_2$ errors scale linearly with prior $W_2$ errors for the Gaussian family and as the square root of prior errors for the two-point family.

The Lipschitz stability result relies on restrictive regularity and essential-oscillation assumptions on the potential. In particular, it does not cover continuous unbounded potentials when the prior has full support. Relaxing these assumptions is an important direction for future research. For learned priors, a further question is whether a target prior can be approximated accurately while maintaining uniform bounds on the Lipschitz constants of the generating maps and on potential oscillation under the generated priors.

\section*{Declaration}

The author acknowledges the following use of AI tools: The numerical experiments are implemented with assistance from OpenAI GPT-5.3-codex in GitHub Copilot. Examples~\ref{ex:likelihood_regularity}, \ref{ex:pth_moment_tail}, and \ref{ex:local_no_holder} were suggested by OpenAI GPT-5.4.  Conversations with OpenAI GPT-5.6 Sol Ultra were important for pinpointing the verifiable, simple regularity conditions sufficient for the Lipschitz stability estimate. In particular, a uniform bound on the $\mu$-essential oscillation of the potential, rather than a global bound, was suggested by OpenAI GPT-5.6 Sol Ultra. Additionally, the extension of the estimate from $p=2$ to $p>1$, specifically the proof strategy of Lemma~\ref{lem:dual_wasserstein_speed}, was found through conversations with OpenAI GPT-5.6 Sol Ultra, and was digested and simplifed by the author. The author takes full responsibility for all the materials included in this manuscript.

\begingroup
\bibliographystyle{siamplain}
\bibliography{references}
\endgroup

\appendix
\section[Proof of the Posterior Coupling Estimate]{Proof of Lemma~\ref{lem:posterior_coupling_estimate}}\label{app:proof_of_posterior_coupling_estimate}

Set $Z_\mu\coloneqq \mu(L)$, $Z_\nu\coloneqq \nu(L)$, $a(x)\coloneqq L(x)/Z_\mu$ and $b(x')\coloneqq L(x')/Z_\nu$. Define
\begin{equation*}
 \gamma_0(\mathrm{d}x, \mathrm{d}x')\coloneqq \min\{a(x),b(x')\}\,\pi(\mathrm{d}x, \mathrm{d}x').
\end{equation*}
Since $\pi$ has marginals $\mu$ and $\nu$, both $a(x)\pi(\mathrm{d}x, \mathrm{d}x')$ and $b(x')\pi(\mathrm{d}x, \mathrm{d}x')$ have total mass one. Thus $\gamma_0$ is a subprobability measure whose marginals are dominated by the two posteriors. For Borel $A\subseteq X$, the residuals are
\begin{equation*}
 \eta_\mu(A)
 \coloneqq \int_{A\times X}(a(x)-b(x'))_+\,\pi(\mathrm{d}x, \mathrm{d}x'),
 \qquad
 \eta_\nu(A)
 \coloneqq \int_{X\times A}(b(x')-a(x))_+\,\pi(\mathrm{d}x, \mathrm{d}x'),
\end{equation*}
where $(t)_+\coloneqq\max\{t,0\}$. They have the same total mass
\begin{equation*}
 r\coloneqq \eta_\mu(X)=\eta_\nu(X)
   =1-\gamma_0(X\times X).
\end{equation*}
If $r>0$, then $\gamma\coloneqq \gamma_0+r^{-1}\eta_\mu\otimes\eta_\nu$ couples $\mathcal B_L(\mu)$ and $\mathcal B_L(\nu)$; if $r=0$, take $\gamma\coloneqq \gamma_0$.

Since $a,b\leq\zeta^{-1}$ and $\pi$ is $W_p$-optimal,
\begin{equation*}
 \int_{X\times X}d_X(x,x')^p\,\gamma_0(\mathrm{d}x, \mathrm{d}x')
 \leq\frac{W_p(\mu,\nu)^p}{\zeta}.
\end{equation*}
For $r>0$, the residual cost satisfies
\begin{align*}
 &\frac1r\int_{X\times X}d_X(x,x')^p\,
 (\eta_\mu\otimes\eta_\nu)(\mathrm{d}x, \mathrm{d}x')\\
 &\quad\leq
 2^{p-1}\left(
 \int_X d_X(x,x_*)^p\,\eta_\mu(\mathrm{d}x)
 +\int_X d_X(x',x_*)^p\,\eta_\nu(\mathrm{d}x')\right)\\
 &\quad\leq
 2^{p-1}\int_{X\times X}
 \bigl(d_X(x,x_*)^p+d_X(x',x_*)^p\bigr)
 |a(x)-b(x')|\,\pi(\mathrm{d}x, \mathrm{d}x').
\end{align*}
The same resulting bound holds when $r=0$, with the residual term absent. Consequently,
\begin{equation}\label{eq:wp_coupling_estimate}
 W_p\bigl(\mathcal B_L(\mu),\mathcal B_L(\nu)\bigr)^p
 \leq\frac{W_p(\mu,\nu)^p}{\zeta}
 +2^{p-1}\int_{X\times X}
 \bigl(d_X(x,x_*)^p+d_X(x',x_*)^p\bigr)
 |a(x)-b(x')|\,\pi(\mathrm{d}x, \mathrm{d}x').
\end{equation}

Using $Z_\mu,Z_\nu\geq\zeta$ and $0\leq L\leq1$, we have
\begin{equation}\label{eq:ab_bound}
 |a(x)-b(x')|
 \leq\frac{|L(x)-L(x')|}{\zeta}
 +\frac{|Z_\mu-Z_\nu|}{\zeta^2},
\end{equation}
while
\begin{equation*}
 |Z_\mu-Z_\nu|
 \leq\int_{X\times X}|L(x)-L(x')|\,\pi(\mathrm{d}x, \mathrm{d}x').
\end{equation*}
Substitution into \eqref{eq:ab_bound} and then \eqref{eq:wp_coupling_estimate}, together with
\begin{equation*}
 \int_{X\times X}\bigl(d_X(x,x_*)^p+d_X(x',x_*)^p\bigr)\,\pi(\mathrm{d}x,\mathrm{d}x')
 =m_p(\mu)+m_p(\nu),
\end{equation*}
gives the claimed estimate.

\section[Supplementary Lemmas for Lipschitz Stability]{Supplementary Lemmas for Theorem~\ref{thm:lipschitz_stability}}\label{app:supplementary_lemmas}

The proof of Theorem~\ref{thm:lipschitz_stability} uses the following two lemmas. The centering inequality controls the covariance obtained by differentiating the tilted measures, and the dual criterion for $W_p$-Lipschitz curves converts this derivative bound into a bound on the distance between the measures.

\begin{lemma}[Sharp Centering Inequality]
\label{lem:sharp_centering}
Let $p>1$ and $p'=p/(p-1)$. Define
\begin{equation*}
     \gamma_p
 :=\max_{0\leq s\leq1}
 \bigl(s^{p-1}+(1-s)^{p-1}\bigr)^{1/p}
 \bigl(s^{p'-1}+(1-s)^{p'-1}\bigr)^{1/p'}.
\end{equation*}
For every $\mu\in\mathcal{P}(X)$ and
$f\in L^{p'}_{\mu}(X)$,
\begin{equation}
\label{eq:sharp_centering}
 \|f-\mu(f)\|_{L^{p'}_{\mu}(X)}
 \leq\gamma_p\inf_{a\in\mathbb R}
 \|f-a\|_{L^{p'}_{\mu}(X)},
\end{equation}
where $\gamma_2=1$ and $\gamma_p<2$.
\end{lemma}

\begin{proof}
Apply the centering bound in \cite[equations~(1.7)--(1.10)]{shargorodsky2023sharp} with exponent $p'$ to $f-a$, for any $a\in\mathbb R$. Since $(f-a)-\mu(f-a)=f-\mu(f)$ and the constant is invariant under conjugating the exponent, taking the infimum over $a$ proves \eqref{eq:sharp_centering}. The same reference gives $\gamma_2=1$ and $\gamma_p\leq2^{|1-2/p|}<2$ for $p>1$.
\end{proof}

The proof of the next lemma adapts the Hopf--Lax duality argument in
\cite[proof of Theorem~3.5]{gigli2015continuity} to arbitrary $p>1$, using
the power-cost estimates in \cite[Section~3]{ambrosio2015sobolev}.

\begin{lemma}[Dual Criterion for $W_p$-Lipschitz Curves]
\label{lem:dual_wasserstein_speed}
Let $(X,d_X)$ be a complete and separable metric space, let $p>1$, and set $p'\coloneqq p/(p-1)$. Let $(\rho_t)_{t\in[0,1]}\subset\mathcal P_p(X)$ satisfy the following
conditions:
\begin{enumerate}[label=(\roman*)]
 \item there exist $\vartheta\in\mathcal P(X)$ and $C<\infty$ such that $\rho_t\leq C\vartheta$ as measures for every $t\in[0,1]$;
 \item there exists $\Lambda\geq0$ such that, for every bounded Lipschitz function $f:X\to\mathbb R$, the map $t\mapsto\rho_t(f)$ is absolutely continuous on $[0,1]$ and
 \begin{equation*}
     \left|\frac{\mathrm d}{\mathrm dt}\rho_t(f)\right|
     \leq\Lambda\bigl\||\mathrm Df|\bigr\|_{L^{p'}_{\rho_t}(X)}
     \quad\text{for almost every }t\in(0,1).
 \end{equation*}
\end{enumerate}
Then $t\mapsto\rho_t$ is $\Lambda$-Lipschitz from $[0,1]$ into $(\mathcal P_p(X),W_p)$; equivalently,
\begin{equation}
 \label{eq:dual_speed_bound}
 W_p(\rho_s,\rho_t)\leq\Lambda(t-s),
 \qquad 0\leq s\leq t\leq1.
\end{equation}
\end{lemma}

\begin{proof}
We first outline the proof strategy. Fix $0\leq s<t\leq1$, set $h\coloneqq t-s$, $A\coloneqq h\Lambda$, and $\mu_r\coloneqq\rho_{s+hr}$ for $r\in[0,1]$.
For a bounded Lipschitz function $\varphi:X\to\mathbb R$, define
\begin{equation*}
 \mathsf Q_r\varphi(x)
 \coloneqq\inf_{x'\in X}\left\{
 \varphi(x')+\frac{d_X(x,x')^p}{p r^{p-1}}\right\},
 \qquad r>0,
 \qquad \mathsf Q_0\varphi\coloneqq\varphi.
\end{equation*}
Kantorovich duality for the cost $d_X^p/p$
\cite[equation~(2.3)]{ambrosio2015sobolev} gives
\begin{equation*}
 \frac1pW_p^p(\mu_0,\mu_1)
 =\sup_{\varphi\text{ bounded Lipschitz}}
 \{\mu_1(\mathsf Q_1\varphi)-\mu_0(\varphi)\}.
\end{equation*}
It therefore suffices to prove that $G(r)\coloneqq\mu_r(\mathsf Q_r\varphi)$ is absolutely continuous with $G'(r)\leq A^p/p$ almost everywhere, for every such $\varphi$. The Hamilton--Jacobi inequality for $\mathsf Q_r\varphi$ provides a negative term involving the pointwise local Lipschitz constant, which absorbs the contribution from the changing measure through assumption~(ii) and Young's inequality. We first use uniform domination to justify this calculation with time-dependent test functions.

Assumption~(ii) implies that $r\mapsto\mu_r(f)$ is absolutely continuous for every bounded Lipschitz $f$. Thus $(\mu_r)_{r\in[0,1]}$ is narrowly continuous, since bounded Lipschitz functions determine narrow convergence. Uniform domination then gives continuity against every bounded Borel function $g$. Indeed, if $r_n\to r$, choose $g_k\in C_b(X)$ converging to $g$ in $L^1_{\vartheta}(X)$. Then
\begin{equation*}
 \limsup_{n\to\infty}|\mu_{r_n}(g)-\mu_r(g)|
 \leq2C\|g-g_k\|_{L^1_{\vartheta}(X)}\longrightarrow0
 \qquad\text{as }k\to\infty.
\end{equation*}
Applying this observation to the bounded Borel function $|\mathrm Df|^{p'}$ shows that $u\mapsto\bigl\||\mathrm Df|\bigr\|_{L^{p'}_{\mu_u}(X)}$ is continuous for each fixed bounded Lipschitz $f$. Rescaling and integrating the bound in assumption~(ii) gives
\begin{equation*}
 |\mu_b(f)-\mu_a(f)|
 \leq A\int_a^b
 \bigl\||\mathrm Df|\bigr\|_{L^{p'}_{\mu_u}(X)}\,\mathrm du,
 \qquad 0\leq a\leq b\leq1.
\end{equation*}

Fix $\varphi$ as above. The Hopf--Lax estimates in \cite[Propositions~11 and~13, Theorem~14, and equation~(3.11)]{ambrosio2015sobolev} show that $r\mapsto\mathsf Q_r\varphi$ is Lipschitz in the supremum norm, that $\operatorname{Lip}(\mathsf Q_r\varphi)\leq p\operatorname{Lip}(\varphi)$, and that, for each $x\in X$, the time derivative exists outside a countable set of times and satisfies
\begin{equation*}
 \dot{\mathsf Q}_r\varphi(x)
 +\frac1{p'}|\mathrm D\mathsf Q_r\varphi|^{p'}(x)\leq0.
\end{equation*}
Here $|\mathrm D\mathsf Q_r\varphi|(x)$ denotes the pointwise local Lipschitz constant of $\mathsf Q_r\varphi$ at $x$, with $r$ fixed, as defined in Section~\ref{sec:lipschitz_stability}. The integrated bound above also gives
\begin{equation*}
 |G(b)-G(a)|
 \leq\|\mathsf Q_b\varphi-\mathsf Q_a\varphi\|_\infty
 +A p\operatorname{Lip}(\varphi)(b-a),
 \qquad 0\leq a<b\leq1,
\end{equation*}
so $G$ is Lipschitz.

We next apply the Hamilton--Jacobi inequality at common times. The map $(r,x)\mapsto\mathsf Q_r\varphi(x)$ is jointly continuous. Rational difference quotients show that the set where its time derivative exists is Borel and that the derivative, extended by zero elsewhere, is bounded and Borel. Computing the pointwise local Lipschitz constant over a countable dense subset of $X$ likewise shows that $(r,x)\mapsto|\mathrm D\mathsf Q_r\varphi|(x)$ is Borel. Fubini's theorem therefore gives derivative existence and the Hamilton--Jacobi inequality for $\vartheta$-almost every $x$, for almost every $r\in(0,1)$. Since $\mu_r\leq C\vartheta$, both statements also hold $\mu_r$-almost everywhere at each such time.

Fix one of these times $r$ at which $G'(r)$ also exists, and set
\begin{equation*}
 z_r\coloneqq\bigl\||\mathrm D\mathsf Q_r\varphi|\bigr\|_{L^{p'}_{\mu_r}(X)},
 \qquad
 q_\varepsilon\coloneqq
 \frac{\mathsf Q_{r+\varepsilon}\varphi-\mathsf Q_r\varphi}{\varepsilon},
 \qquad 0<\varepsilon<1-r.
\end{equation*}
We separate the changes in the test function and the measure:
\begin{equation*}
 \frac{G(r+\varepsilon)-G(r)}{\varepsilon}
 =\mu_{r+\varepsilon}(q_\varepsilon)
 +\frac{(\mu_{r+\varepsilon}-\mu_r)(\mathsf Q_r\varphi)}{\varepsilon}.
\end{equation*}
The functions $q_\varepsilon$ are uniformly bounded and converge to $\dot{\mathsf Q}_r\varphi$ $\vartheta$-almost everywhere, hence in $L^1_{\vartheta}(X)$. Domination and continuity against bounded Borel functions give
\begin{align*}
 |\mu_{r+\varepsilon}(q_\varepsilon)-\mu_r(\dot{\mathsf Q}_r\varphi)|\leq C\|q_\varepsilon-\dot{\mathsf Q}_r\varphi\|_{L^1_{\vartheta}(X)}
 +|\mu_{r+\varepsilon}(\dot{\mathsf Q}_r\varphi)
       -\mu_r(\dot{\mathsf Q}_r\varphi)|\longrightarrow0.
\end{align*}
For the second term, the exceptional set in assumption~(ii) may depend
on the test function. We therefore use its integrated form with
$f=\mathsf Q_r\varphi$ held fixed. Continuity of the
$L^{p'}_{\mu_u}(X)$-norm of $|\mathrm D\mathsf Q_r\varphi|$ gives
\begin{equation*}
 \limsup_{\varepsilon\downarrow0}
 \frac{(\mu_{r+\varepsilon}-\mu_r)(\mathsf Q_r\varphi)}{\varepsilon}
 \leq\lim_{\varepsilon\downarrow0}\frac A\varepsilon
 \int_r^{r+\varepsilon}
 \bigl\||\mathrm D\mathsf Q_r\varphi|\bigr\|_{L^{p'}_{\mu_u}(X)}\,\mathrm du
 =A z_r.
\end{equation*}
Combining these bounds, integrating the Hamilton--Jacobi inequality
against $\mu_r$, and applying Young's inequality yields
\begin{equation*}
 G'(r)\leq\mu_r(\dot{\mathsf Q}_r\varphi)+A z_r
 \leq-\frac{z_r^{p'}}{p'}+A z_r
 \leq\frac{A^p}{p}.
\end{equation*}
This holds for almost every $r$. Since $G$ is Lipschitz,
\begin{equation*}
 \mu_1(\mathsf Q_1\varphi)-\mu_0(\varphi)
 =G(1)-G(0)\leq\frac{A^p}{p}.
\end{equation*}
Taking the supremum over $\varphi$ in the dual formula gives
$W_p(\rho_s,\rho_t)=W_p(\mu_0,\mu_1)\leq A=\Lambda(t-s)$.
The case $s=t$ is immediate.
\end{proof}

\Needspace{6\baselineskip}
\enlargethispage{2\baselineskip}
\section{Details of the Numerical Example}\label{app:darcy_details}

\subsection{Stability Constant}\label{app:darcy_constants}

For the model in Section~\ref{sec:numerical_example}, set
\begin{equation*}
 d_y\coloneqq\max\{|y-1/3|,|y-2/3|\},\qquad
 M\coloneqq\frac{d_y^2}{2\sigma^2},\qquad
 \overline K\coloneqq\frac{\sqrt2\,d_y}{8\sigma^2}.
\end{equation*}
Then $0\leq\Phi\leq M$ and $\operatorname{Lip}(\Phi)\leq \overline K$. With $C_{\mathrm P,2}(\mu_h)\leq5/4$, Theorem~\ref{thm:lipschitz_stability} gives
\begin{equation*}
 C_{\mathrm{stab}}=e^{M/2}+ \overline K\sqrt{5/4}\,e^M\approx7.0867.
\end{equation*}
This bound can be conservative: at $\sigma=0.05$ it gives $C_{\mathrm{stab}}\approx2.05\times10^5$, whereas the computed ratio $W_2(\rho_0,\rho_h)/h$ at $h=10^{-4}$ is approximately $0.676$.

\subsection{Posterior Wasserstein Distance}\label{app:posterior_distance}

The posterior distances can be computed by one-dimensional quadrature. Let $F_h$ be the cumulative distribution function of the latent posterior with density
\begin{equation*}
 f_h(z)\coloneqq\frac{L(zv_h)e^{-z^2/2}}{Z_{\mu_h}\sqrt{2\pi}},
 \qquad z\in\mathbb R.
\end{equation*}
Since $v,w$ are orthonormal, $\|zv-z'v_h\|_X^2=(z-z')^2+h^2(z')^2$. The second term depends only on the second marginal. Monotone coupling therefore minimizes the cost \cite[Theorem~2.9]{santambrogio2015optimal} and gives
\begin{equation}\label{eq:darcy_posterior_quantiles}
 W_2(\rho_0,\rho_h)^2
 =\int_0^1\left[
 \bigl(F_0^{-1}(t)-F_h^{-1}(t)\bigr)^2
 +h^2\bigl(F_h^{-1}(t)\bigr)^2\right]\,\mathrm dt.
\end{equation}
The second term in \eqref{eq:darcy_posterior_quantiles} is $h^2\int_{\mathbb R}z^2f_h(z)\,\mathrm dz$, whose coefficient tends to a positive limit as $h\downarrow0$ by dominated convergence. Together with \eqref{eq:darcy_posterior_bound}, this gives $W_2(\rho_0,\rho_h)\asymp h$.

We compute $F_h$ by the composite trapezoidal rule on $80\,001$ equally spaced points in $[-10,10]$ and invert it by linear interpolation. We evaluate \eqref{eq:darcy_posterior_quantiles} by the composite Simpson rule after substituting $t=F_{\mathcal N}(r)$, where $F_{\mathcal N}$ is the standard normal cumulative distribution function, using $32\,001$ equally spaced points in $[-7.5,7.5]$. At $h=0$, Simpson quadrature of the permeability and its centered square against $f_0$ on the latent grid gives the reference mean $\rho_0(Q)\approx1.4565$ and standard deviation $\operatorname{sd}_{\rho_0}(Q)\approx0.2411$. Halving both grid spacings and then enlarging the intervals to $[-12,12]$ and $[-8,8]$ changed the posterior distances by less than $10^{-7}$ relatively in each comparison.

\subsection{Two-Point Priors}\label{app:darcy_atomic}

Write $\ell(t)\coloneqq L(tv)$ and $a_h\coloneqq\ell(h)/(\ell(h)+\ell(1))$. Since $G(tv)=1/2-\tanh(t/\sqrt2)/6$ and $y>1/2$, the likelihood $\ell$ decreases on $[0,\infty)$, so $a_h<a_0$ for $h>0$. Differentiating at zero gives
\begin{equation*}
 a'_0=-\frac{a_0(1-a_0)(y-1/2)}{6\sqrt2\,\sigma^2}<0.
\end{equation*}
Monotone transport gives
\begin{equation}\label{eq:darcy_atomic_error}
 W_2(\nu_0,\nu_h)=\frac{h}{\sqrt2},
 \qquad
 W_2\bigl(\mathcal B_L(\nu_0),\mathcal B_L(\nu_h)\bigr)^2
 =a_hh^2+(a_0-a_h)\sim-a'_0h.
\end{equation}
The priors are supported in the unit ball and have evidence at least $e^{-M}$. Their Poincar\'e constants are infinite.

\end{document}